\documentclass[11pt]{article}
\usepackage{epsfig,amsmath,amsthm,latexsym,graphicx,amsfonts,amssymb}
\newcommand{\beq}{\begin{equation}}
\newcommand{\eeq}{\end{equation}}
\newcommand{\R}{\mathbb R}

\newcommand{\E}{\mathbb E}
\newcommand{\Sr}{{\mathbb S}_r^{q-1}}
\newcommand{\ex}{\mu}
\newcommand{\hf}{\frac{1}{2}}
\newcommand{\fdelta}{\delta}
\newcommand{\feps}{\epsilon}

\newcommand{\sign}{\operatorname{\mathrm{sign}}}
\newcommand{\signmap}{\operatorname{\mathrm{sign}_X}}

\newcommand{\signmapZ}{\operatorname{\mathrm{sign}_Z}}
\newcommand{\signmapinv}{\operatorname{\mathrm{sign}_X^{-1}}}
\newcommand{\signmapeps}{\operatorname{\mathrm{sign}_{X\fdelta}}}

\newcommand{\signmapepsinv}{\operatorname{\mathrm{sign}_{X\fdelta}^{-1}}}
\newcommand{\arctanh}{\operatorname{\mathrm{arctanh}}}
\newcommand{\tr}{\operatorname{\mathrm{tr}}}
\newcommand{\vol}{\operatorname{\mathrm{Vol}}}
\newcommand{\volqeucl}{\operatorname{\mathrm{Vol}_q}}
\newcommand{\volqeuclsq}{\operatorname{\mathrm{Vol}_q^2}}
\newcommand{\col}{\operatorname{\mathrm{col}}}
\newcommand{\comp}{\operatorname{\mathrm{Comp}}}
\newcommand{\cy}{\mathcal{Y}}
\newcommand{\SX}{\mathcal{S}_X}
\newcommand{\SbarX}{\mathcal{S}_{\overline{X}}}
\newcommand{\SY}{\mathcal{S}_{Y}}
\newcommand{\SZ}{\mathcal{S}_{Z}}
\newcommand{\SZI}{\mathcal{S}_{Z_I}}
\newcommand{\SXI}{\mathcal{S}_{X_I}}
\newcommand{\defeq}{\stackrel{\mbox{\tiny{def}}}{=}}
\newcommand{\dmin}{\operatorname{\Delta_{\mathrm{min}}}(X)}
\newcommand{\dmax}{\operatorname{\Delta_{\mathrm{max}}}(X)}
\newtheorem{theorem}{Theorem}
\newtheorem{corol}[theorem]{Corollary}
\newtheorem{lemma}[theorem]{Lemma}

\newtheorem{defn}{Definition}

\begin{document}

\title{Volumes of logistic regression models \\ with applications to model selection}
\author{James G. Dowty}
\date{\today}

\maketitle

\abstract{Logistic regression models with $n$ observations and $q$ linearly-independent covariates are shown to have Fisher information volumes which are bounded below by $\pi^q$ and above by ${n \choose q} \pi^q$.  This is proved with a novel generalization of the classical theorems of Pythagoras and de Gua, which is of independent interest.  The finding that the volume is always finite is new, and it implies that the volume can be directly interpreted as a measure of model complexity.  The volume is shown to be a continuous function of the design matrix $X$ at generic $X$, but to be discontinuous in general.  This means that models with sparse design matrices can be significantly less complex than nearby models, so the resulting model-selection criterion prefers sparse models.  This is analogous to the way that $\ell^1$-regularisation tends to prefer sparse model fits, though in our case this behaviour arises spontaneously from general principles.  Lastly, an unusual topological duality is shown to exist between the ideal boundaries of the natural and expectation parameter spaces of logistic regression models.
}

\section{Overview and context of results}
\label{S:overview}

Any full-rank, $q \times n$ matrix $X$ with $q \le n$ is the design matrix of a unique logistic regression model $\SX$ for binary data $y \in \{ 0,1 \}^n$ \cite{McCullaghNelder83}.  Here, the $n$ components of $y$ are considered to be draws from $n$ independent Bernoulli random variables and we are using the canonical link function.

When equipped with the Fisher information metric, the $q$-dimensional parameter space of $\SX$ becomes a Riemannian manifold \cite{KassVos97}. Further, by Chentsov's theorem \cite{Chentsov78,AyEtAl12}, the Fisher information metric is the only natural metric on $\SX$, in the sense that it is the only metric which is invariant under natural statistical transformations related to sufficient statistics.  The geometry of $\SX$ is therefore likely to be important and useful in understanding the behaviour of $\SX$.

In this paper, we concentrate on the simplest geometric invariant of $\SX$, namely its volume $\vol(\SX)$.  We show that $\vol(\SX)$ is always finite, which was previously unknown, and we prove the following bounds.

\begin{theorem} \label{T:volSXbound}
$$ \pi^q \le \vol(\SX) \le {n \choose q} \pi^q.$$
\end{theorem}

These bounds are based on Theorem \ref{T:deGua}, which is a novel generalisation of the classical theorems of Pythagoras and de Gua
\cite[p. 207]{Wells91}
and is of independent interest.

Our result that $\vol(\SX)$ is finite has a number of theoretical consequences for the logistic regression model $\SX$, since it shows that $\SX$ satisfies the common regularity condition that its Jeffreys prior should be proper.  One consequence of this is that $\vol(\SX)$ can be directly interpreted as a measure of model complexity, since a simple, monotonic function of $\vol(\SX)$ then approximates the parametric complexity for large $n$
\cite{Rissanen96}\cite[eqn. 2.21]{Grunwald05}.
Here, the parametric complexity is an information-theoretic measure of the statistical size of $\SX$ which can be
subtracted from
the maximized log-likelihood to give a natural measure of the parsimony of $\SX$ as a model for data $y$ \cite[eqn. 2.20]{Grunwald05}.  The corresponding model-selection criterion is known as the minimum description length (MDL) criterion \cite{BarronEtAl98,Rissanen07}
and it has many desirable properties, such as almost sure consistency for parametric models and the ability to select a data-generating model from a countable set of models for all sufficiently large $n$ with probability $1$ \cite{BarronCover91}.

No previous logistic regression studies have used the volume as a measure of model complexity, though a few studies have used other variants of MDL:  \cite{HansenYu02} used a mixture MDL approach \cite{HansenYu01} in which a normal prior was placed on the regression coefficients and MDL principles were used to choose the hyper-parameters; \cite{ZhouEtAl05} and \cite{QianField02} were based on the approximation of \cite{QianKunsch98} and its $2$-part code approach; and \cite{FowlerLindblad11} used a renormalized NML criterion \cite{Rissanen01} adapted from linear regression to logistic regression with a weighting method.

The above connections with MDL show that $\vol(\SX)$ is an important measure of model complexity, but we also show that it has some remarkable geometric properties.  Perhaps the strangest and most useful property is that $\vol(\SX)$ is a discontinuous function of $X$.  Some design matrices, such as those with some rows consisting only of zeroes, are significantly less complex than nearby design matrices.  This means that a model-selection criterion based on $\vol(\SX)$ will tend to choose models with sparse design matrices over models with design matrices with many small entries.  This behaviour is analogous to (though different from) the way that $\ell^1$-regularised regression models tend to choose model fits with coefficients equal to $0$ over model fits with small coefficients \cite{Tibshirani96,Tibshirani11}.

We derive an approximation to $\vol(\SX)$ under the mild assumptions that $n$ is large, the rows of $X$ are realisations of independent and identically distributed (IID) random variables and $X$ has full rank with probability $1$, plus a more technical condition on the covariate distribution (see Section \ref{S:approxvol}).  This approximation to $\vol(\SX)$ then gives the following model-selection criterion.

\begin{defn}[Approximate volume criterion]
\label{D:volcrit}
Given a countable set of competing logistic regression models for binary data $y \in \{ 0,1 \}^n$ with $n$ observations, the {\em approximate volume criterion} advocates choosing the model $\SX$ with the smallest value of
\beq
\label{E:approxvolcritintro}
- \log p(y | \hat{\beta}(y)) + \frac{q}{2}\log \frac{\pi}{2} + \frac{1}{2}\log {n-n_0 \choose q}
\eeq
where $\log p(y | \hat{\beta}(y))$ is the maximized log-likelihood and the design matrix of $\SX$ has
$n$ rows, $q$ columns
and exactly $n_0$ rows with all entries equal to $0$.
\end{defn}

The main result of \cite{QianField02} implies that this criterion is strongly consistent, meaning that it will select the correct model almost surely as the sample size $n$ goes to infinity.  As a proof of principle, we apply this model-selection criterion to a simulated image processing problem, giving promising results (see Figure \ref{F:imageproc}).
Our approach to this problem couples the approximate volume criterion with $\ell^1$-regularisation \cite{Tibshirani96,Tibshirani11}, making our results applicable to the case $q > n$ where the number of potential covariates is larger than the number of observations (see Section \ref{S:imageproc}).

Lastly, we consider the behaviour of the logistic regression model $\SX$ for large parameter values when $X$ is generic, meaning that any $q$ of the rows of $X$ are linearly independent.  We first show that, while $\vol(\SX)$ is a discontinuous function of $X$ in general, it is continuous at generic $X$.  This raises the possibility that a closed-form expression for $\vol(\SX)$ might exist for generic $X$.  We second consider the relationship between two natural polygonal decompositions of the ideal boundaries of the natural and expectation parameter spaces of $\SX$.  The expectation parameter space is an open polytope, so its ideal boundary (the boundary of its closure) decomposes into lower-dimensional polytopes, while the ideal boundary of the natural parameter space (approximated by a sphere of large radius $r$ centred at the origin) is divided into spherical polytopes by the hyperplanes $\{ \beta \in \R^q \mid x_i \beta = 0 \}$, where each $x_i$ is a row of $X$.  We show that these two polygonal decompositions are topologically dual via the reparameterisation map, meaning that this function approximately maps $k$-dimensional polytopes in the $(q-1)$-dimensional boundary of one parameter space to $(q-1-k)$-dimensional polytopes in the boundary of the other, with this approximation becoming exact as the radius $r$ goes to infinity (see Figure \ref{F:duality}).  This highly unusual behaviour is interesting in its own right, but it also has implications for the computation of $\vol(\SX)$ (see the end of Section \ref{S:duality}).

The rest of this paper is set out as follows.  In Section \ref{S:full} we describe a model which is geometrically a Euclidean cube and into which all logistic regression models for $n$ observations can be isometrically embedded.  We then calculate the Fisher information metric of a logistic regression model $\SX$ and show that the corresponding volume $\vol(\SX)$ is unchanged by rescaling the covariates (Section \ref{S:logreg}).  In Section \ref{S:finite}, we use the embedding of $\SX$ into the Euclidean cube to prove Theorem \ref{T:volSXbound}.  We then show that $\vol(\SX)$ is a discontinuous function of $X$ (Section \ref{S:discont}) before deriving the approximate volume criterion of Definition \ref{D:volcrit} and applying it to an image processing problem (Section \ref{S:compl}).  We show that $\vol(\SX)$ is continuous at generic $X$ and prove the above topological duality in Section \ref{S:largebeta}.  We then describe some of the discontinuities in $\vol(\SX)$ which can occur at non-generic $X$ in Section \ref{S:nongeneric}, before finishing with some concluding remarks in Section \ref{S:conc}.

\section{The saturated model for binary data}
\label{S:full}

In this section, we introduce a statistical model into which all logistic regression models with $n$ observations can be isometrically embedded (though we will not describe the embedding until Section \ref{S:finite}).

Consider binary data $y \in \{ 0,1\}^n$ with components $y_1, \ldots, y_n$ which are realizations of $n$ independent random variables $Y_1, \ldots, Y_n$.  The most general stochastic model for this data, which we call the {\em saturated model}, has a separate model parameter for each observation.  One parameterisation for this model is in terms of a parameter $\ex \in (0,1)^n$ interpreted as the probability $\ex_i = P(Y_i=1) = \E Y_i$.  The likelihood function for this parameterisation is therefore
$$ \prod_{i=1}^n \ex_i^{y_i} (1 - \ex_i)^{1 - y_i}. $$
Alternatively, we can parameterise the saturated model with the log-odds parameter $\lambda \in \R^n$ which is related to the parameter $\ex$ by
\beq\label{E:lambda}
\lambda_i = \log \left( \frac{\ex_i}{1-\ex_i} \right) \mbox{ or, equivalently, } \ex_i = \frac{\exp(\lambda_i)}{1 + \exp(\lambda_i)}.
\eeq
The log-odds parameterisation is of particular interest to us because each logistic regression model is a stochastic model of the above form with the log-odds constrained to lie in a linear subspace.

From (\ref{E:lambda}), $1-\ex_i = (1 + \exp(\lambda_i))^{-1}$, so the log-likelihood for the log-odds parameterisation is
\begin{eqnarray}
\ell(\lambda)
&=&\nonumber \log \left( \prod_{i=1}^n \left( \frac{\exp(\lambda_i)}{1 + \exp(\lambda_i)}\right)^{y_i}
\left( \frac{1}{1 + \exp(\lambda_i)}\right)^{1-y_i}  \right)\\
&=& y^T \lambda - \sum_{i=1}^n \log \left( 1 + \exp(\lambda_i)\right), \label{E:expfamfull}
\end{eqnarray}
where $y$ and $\lambda$ are interpreted as column matrices in (\ref{E:expfamfull}).
This shows that the saturated model for binary data is an exponential family \cite[\S 2.2]{KassVos97} and that $y$ is a natural sufficient statistic with $\lambda$ the corresponding natural parameter.  Since $\ex_i = \E Y_i$, $\ex$ is the expected value of the sufficient statistic, so $\ex$ is the corresponding expectation parameter for the exponential family.

Recall that the Fisher information metric of a stochastic model with parameter space $U$ is a Riemannian metric $g_U$ on $U$ given by either of the following expressions
\beq
\label{E:fisherinfodefn}
g_U = \E[(\nabla \ell)(\nabla \ell)^T] = - \E[\mbox{Hess}(\ell)]
\eeq
where $U \subseteq \R^q$ is an open set, $\ell:U \to \R$ is the log-likelihood function, $\nabla \ell$ is the gradient of $\ell$  (interpreted as a column matrix in (\ref{E:fisherinfodefn})), $\mbox{Hess}(\ell)$ is the Hessian matrix of $\ell$ and the expectation is taken over the observed data \cite[\S 2.2]{AmariNagaoka00}.  The second equality of (\ref{E:fisherinfodefn}) assumes certain regularity conditions, which are satisfied by all models in this paper.  We will sometimes call $g_U$ the Fisher information {\em matrix} of the parameterisation $U$, to distinguish it from the more abstract Fisher information {\em metric} of the stochastic model, which is independent of the parameterisation because all reparameterisation maps are isometries (e.g., by Lemma \ref{L:pullback}).  By Chentsov's theorem \cite{Chentsov78,AyEtAl12}, the Fisher information metric is, in some sense, the only natural metric on a stochastic model.

\begin{lemma}
\label{L:metric_fullmodel_logodds}
The Fisher information matrix for the log-odds parameterisation of the saturated model is a diagonal matrix $D_\lambda$ whose $i^{th}$ diagonal component is
$$ (D_\lambda)_{ii} = \frac{1}{4}\cosh^{-2}(\lambda_i/2).$$
\end{lemma}

\begin{proof}
This follows easily from (\ref{E:expfamfull}) and (\ref{E:fisherinfodefn}), for example
$$(D_\lambda)_{ii} = -\frac{\partial^2 \ell}{\partial \lambda_i^2}
= \frac{\exp(\lambda_i)}{(1 + \exp(\lambda_i))^2} = \frac{1}{(\exp(\lambda_i/2) + \exp(-\lambda_i/2))^2}
= \frac{1}{4\cosh^2 (\lambda_i/2)}.$$
\end{proof}

Lemma \ref{L:metric_fullmodel_logodds} implies that the saturated model is isometric to an $n$-fold product of isometric $1$-dimensional Riemannian manifolds.  But since all $1$-dimensional Riemannian manifolds are Euclidean, this in turn implies that the saturated model is isometric to either a Euclidean cube or Euclidean space.  We now give a parameterisation for the saturated model which realises this isometry.

Let $\Xi$ be the open cube $\Xi \defeq (-\frac{\pi}{2}, \frac{\pi}{2})^n$ and define the parameter $\xi = (\xi_1, \ldots, \xi_n) \in \Xi$ by
\begin{equation}\label{E:xi_param}
\xi_i = \arcsin(2 \ex_i - 1) \mbox{ or, equivalently, } \ex_i = \frac{1}{2}(1 + \sin \xi_i)
\end{equation}
for each $i=1, \ldots, n$.
In light of the following lemma, we will call this the {\em Euclidean parameterisation} of the saturated model.

\begin{lemma}
\label{L:eucl_cube}
The Fisher information matrix $g_\Xi$ for the parameterisation (\ref{E:xi_param}) is the identity matrix everywhere in $\Xi$.  Therefore the saturated model for binary data is isometric to an open, $n$-dimensional Euclidean cube of side-length $\pi$.
\end{lemma}

\begin{proof}
From (\ref{E:xi_param}), $1 - \ex_i
= \frac{1}{2}(1 - \sin \xi_i)$, so the log-likelihood function with respect to the $\xi$ parameterisation is $\ell:\Xi \to \R$ given by
$$ \ell(\xi) = -n \log 2 + \sum_{i=1}^n \log(1 + \epsilon_i \sin \xi_i) $$
where $\epsilon_i = 2 y_i - 1$.  Therefore,
$$\frac{\partial \ell}{\partial \xi_i}
= \frac{\epsilon_i \cos \xi_i}{1 + \epsilon_i \sin \xi_i}$$
and so
$\partial^2 \ell/\partial \xi_i \partial \xi_j = 0$
if $i \not= j$, hence the Fisher information matrix is diagonal.  Also,
\begin{eqnarray*}
\frac{\partial^2 \ell}{\partial \xi_i^2}
&=& - \frac{\epsilon_i \sin \xi_i (1 + \epsilon_i \sin \xi_i)+ \epsilon_i^2 \cos^2 \xi_i}{(1 + \epsilon_i \sin \xi_i)^2} \\
&=& - \frac{\epsilon_i \sin \xi_i + \sin^2 \xi_i + \cos^2 \xi_i}{(1 + \epsilon_i \sin \xi_i)^2} \mbox{ since $\epsilon_i = \pm 1$}\\
&=&  - \frac{1}{1 + \epsilon_i \sin \xi_i}.
\end{eqnarray*}
Now, if $f:\R \to \R$ is any function then $\E [f(\epsilon_i)] = \ex_i f(1) + (1-\ex_i) f(-1)$ by definition of the expectation, so using the relations $\ex_i = \frac{1}{2}(1 + \sin \xi_i)$ and $1 - \ex_i = \frac{1}{2}(1 - \sin \xi_i)$ we have
$$ - \E \left[ \frac{\partial^2 \ell}{\partial \xi_i^2} \right]
= \frac{\frac{1}{2}(1 + \sin \xi_i)}{1 + \sin \xi_i} + \frac{\frac{1}{2}(1 - \sin \xi_i)}{1 - \sin \xi_i} = 1,$$
proving the lemma.
\end{proof}

For future reference we note from (\ref{E:lambda}) and (\ref{E:xi_param}) that
$(1 + \exp(-\lambda_i))^{-1} = \frac{1}{2}(1 + \sin \xi_i) $
so
$$ \sin \xi_i = \frac{2}{1 + \exp(-\lambda_i)} -1 = \frac{1 - \exp(-\lambda_i)}{1 + \exp(-\lambda_i)}
= \frac{\exp(\lambda_i/2) - \exp(-\lambda_i/2)}{\exp(\lambda_i/2) + \exp(-\lambda_i/2)} = \tanh \frac{\lambda_i}{2},$$
hence the Euclidean and log-odds parameterisations are related by
\begin{equation}\label{E:Eucl_logodds}
\xi_i = \arcsin \left(\tanh \frac{\lambda_i}{2} \right)
\end{equation}
where $\arcsin$ has domain $(-1,1)$ and range $(-\frac{\pi}{2},\frac{\pi}{2})$.

\section{Logistic regression models and their volumes}
\label{S:logreg}

Partly to establish our notation, this section recalls the definition of a logistic regression model and its volume before showing that the volume is invariant under re-scaling the covariates.

Here and throughout this paper, let $X$ be a full-rank, real $n \times q$ matrix with $q \le n$.  Given such an $X$, there is a unique logistic regression model $\SX$ which is the sub-model of the saturated model of Section \ref{S:full} whose log-odds parameters $\lambda \in \R^n$ are all of the form
\begin{equation}\label{E:logreg_submodel}
\lambda = X \beta
\end{equation}
for some $\beta \in \R^q$, to be estimated \cite{McCullaghNelder83}.
We consider $\beta$ to be a column matrix and we consider the $i^{th}$ row $x_i$ of $X$ to be a row matrix, so $\lambda_i = x_i \beta$ is a $1 \times 1$ matrix, considered to simply be a real number.

Substituting (\ref{E:logreg_submodel}) into (\ref{E:expfamfull}) shows that $\SX$ is an exponential family with natural parameter $\beta$ and corresponding natural sufficient statistic $X^T y$, where $y \in \{ 0,1 \}^n$ is the observed data.

We now calculate the Fisher information matrix of $\SX$ for the natural parameter space.

\begin{lemma}
\label{L:FImatrix_logreg}
At $\beta \in \R^q$, the Fisher information matrix of the natural parameterisation of $\SX$ is the $q \times q$ matrix
$$X^T D_{X \beta} \,  X$$
where $D$ is the diagonal matrix of Lemma \ref{L:metric_fullmodel_logodds} but is here evaluated at $\lambda = X \beta$.
\end{lemma}

We will prove Lemma \ref{L:FImatrix_logreg} using the following general lemma (which is well-known but proved below because a published proof is not known to the author).

\begin{lemma}
\label{L:pullback}
Let $U$ and $V$ be parameter spaces for two stochastic models and let $\ell_U: U \to \R$ and $\ell_V: V \to \R$ be the corresponding log-likelihood functions.  If $\phi:U \to V$ is a differentiable function and $\ell_U = \ell_V \circ \phi$ then
$$g_U = J^T g_V J$$
where $g_U$ and $g_V$ are the Fisher information matrices of the two parameterisations and $J$ is the Jacobian matrix of $\phi$ (here, $g_U$ and $J$ are evaluated at any $u \in U$ and $g_V$ is evaluated at $\phi(u) \in V$).  In other words, $g_U$ is the pull-back of $g_V$ via $\phi$.
\end{lemma}

\begin{proof}
By (\ref{E:fisherinfodefn}), $g_U = \E[(\nabla \ell_U)(\nabla \ell_U)^{T}]$ and $g_V = \E[(\nabla \ell_V)(\nabla \ell_V)^{T}]$, where $\nabla \ell_U$ and $\nabla \ell_V$ are gradients of $\ell_U$ and $\ell_V$, and recall that $\nabla \ell_U = J^T \nabla \ell_V$.  Therefore
$$ g_U = \E[(\nabla \ell_U)(\nabla \ell_U)^{T}] = \E[J^T (\nabla \ell_V)(\nabla \ell_V)^{T} J]
= J^T \E[(\nabla \ell_V)(\nabla \ell_V)^{T}] J = J^T g_V J$$
as required.
\end{proof}

\begin{proof}[Proof of Lemma \ref{L:FImatrix_logreg}]
Simply apply Lemma \ref{L:pullback} to the case where $U$ is the natural ($\beta$) parameterisation of the logistic regression model, $V$ is the natural (log-odds) parameterisation of the saturated model and $\phi$ is the function $\beta \mapsto X \beta$.  For then $J=X$ and $g_V=D_{X\beta}$ at $\phi(\beta)$ by Lemma \ref{L:metric_fullmodel_logodds}.
\end{proof}

Now, recall that, to any oriented $q$-dimensional Riemannian manifold $M$ with metric tensor $g$, there is a natural volume form, given in local co-ordinates as $\sqrt{\det g}$ times the standard volume form on $\R^q$, and there is a natural notion of the volume of $M$, obtained by integrating this form over $M$ \cite[p. 329-30]{KassVos97}.
So by Lemma \ref{L:FImatrix_logreg}, the {\em volume density} of $\SX$ at a point $\beta \in \R^q$ of the natural parameter space of $\SX$ is
$$ \sqrt{\det(X^T D_{X \beta} \, X)} $$
and the {\em volume} of the logistic regression model $\SX$ is
\begin{equation}\label{E:vol_defn}
\vol(\SX) \defeq  \int_{\R^q} \sqrt{\det(X^T D_{X \beta} \, X)} \, d\beta .
\end{equation}
When $\vol(\SX)$ is finite and non-zero, the Jeffreys prior is proper, and is therefore equal to $\sqrt{\det(X^T D_{X \beta} \, X)}/\vol(\SX)$ at a point $\beta \in \R^q$.

\begin{lemma} \label{L:volzero}
$\vol(\SX) > 0$ if and only if $X$ has rank $q$.
\end{lemma}

\begin{proof}
If $X$ has rank $q$ then $X^T D_{X \beta} \, X$ is a positive definite matrix so it has a strictly positive determinant, hence the volume density is strictly positive everywhere and $\vol(\SX) > 0$.  On the other hand, if $X$ has rank less than $q$ then $\det(X^T D_{X \beta} \, X) = 0$ everywhere, so $\vol(\SX) = 0$.
\end{proof}

The following lemma shows that the volume is invariant under changes to the design matrix $X$, such as rescaling, which do not change its column space $\col(X)$ (recall that the column space of $X$ is the vector subspace of $\R^n$ spanned by the $q$ columns of $X$).

\begin{lemma}
If $X$ and $\overline{X}$ are $n \times q$ matrices with $\col(X) = \col(\overline{X})$ then $\vol(\SX) = \vol(\SbarX)$.
\end{lemma}

\begin{proof}
If $\col(X)$ has dimension less than $q$ then the ranks of $X$ and $\overline{X}$ are both less than $q$ so $\vol(\SX) = \vol(\SbarX)= 0$ by Lemma \ref{L:volzero}.

If $\col(X)$ has dimension $q$ then the columns of $X$ and $\overline{X}$ both form bases for $\col(X)$, so there exists an invertible $q \times q$ matrix $M$ (the change-of-basis matrix) so that $\overline{X} = X M$.  If we set $\bar{\beta} = M^{-1} \beta$ then $\overline{X} \bar{\beta} = X \beta$ and hence $\bar{\ell}(\bar{\beta}) = \ell(\beta)$, where $\bar{\ell}$ and $\ell$ are the two likelihood functions, so Lemma \ref{L:pullback} shows that the two models are isometric and hence have the same volumes.
Alternatively, it is not hard to show $\vol(\SX) = \vol(\SbarX)$ directly by effecting a change of variables $\bar{\beta} = M^{-1} \beta$ in the definition (\ref{E:vol_defn}).
\end{proof}

\section{Bounds on $\vol(\SX)$}
\label{S:finite}

This section establishes the volume bounds of Theorem \ref{T:volSXbound} and proves a generalisation of Pythagoras' and de Gua's theorems along the way.

As above, let $X$ be a real, full-rank, $n \times q$ matrix with $q \le n$, let $\SX$ be the corresponding logistic regression model and let $\Xi$ be the Euclidean parameter space of the saturated model with $n$ observations.
Define $\phi: \R^q \to \Xi$ by $\phi = (\phi_1, \ldots, \phi_n)$ where
\beq \label{E:phi}
\phi_i(\beta) = \arcsin \left(\tanh \frac{x_i \beta}{2} \right)
\eeq
and $x_i$ is the $i^{th}$ row of $X$ (recall that $x_i$ is a row matrix and $\beta$ is a column matrix so $x_i \beta$ is a $1 \times 1$ matrix, i.e., a real number).  As in the comment following (\ref{E:Eucl_logodds}), we take $\arcsin$ in (\ref{E:phi}) to have domain $(-1,1)$ and range $(-\frac{\pi}{2},\frac{\pi}{2})$.
When the design matrix $X$ is not clear from the context, we will write $\phi_X$ instead of $\phi$.

By (\ref{E:Eucl_logodds}) and (\ref{E:logreg_submodel}), $\phi$ maps the natural parameter space of $\SX$ into the Euclidean parameter space $\Xi$ of the saturated model in a way which respects likelihoods.  So by Lemma \ref{L:pullback}, $\phi$ is a local isometry onto its image.  We will show that $\phi$ is injective, so it will follow that $\vol(\SX)$ is the $q$-dimensional Euclidean volume (i.e., Hausdorff measure) of the image $\phi(\R^q)$ of $\phi$ inside the Euclidean cube $\Xi$.  This does not guarantee that $\vol(\SX)$ is finite, however, since an infinitely long curve can be embedded into a finite cube by spiraling around a circle, for example.  So in Lemma \ref{L:monotonic} we will show that the embedding $\phi$ does not exhibit such non-monotonic behaviour.  We will then use a novel generalization of Pythagoras' and de Gua's theorems (Lemma \ref{L:deGua} and Theorem \ref{T:deGua}) to bound the volumes of logistic regression models (Theorem \ref{T:volSXbound}).  In particular, this will imply that $\vol(\SX)$ is always finite.

We begin with the generalization of Pythagoras' and de Gua's theorems.  For any set $I \subseteq \{ 1, \ldots, n \}$ with $q$ elements, say $I = \{ i_1, \ldots, i_q \}$ where $i_1 < \dots < i_q$, define $\rho_I:\R^n \to \R^q$ to be the projection of $\R^n$ onto those co-ordinates with indices in $I$, i.e., let $\rho_I$ be the $q \times n$ matrix so that $\rho_I [\xi_1 \, \ldots \, \xi_n]^T = [\xi_{i_1} \, \ldots \, \xi_{i_q}]^T$ for any column matrix $\xi \in \R^n$.

\begin{lemma}
\label{L:deGua}
If $V$ is any $n \times q$ matrix then
\beq \label{E:deGua_eq}
\det(V^T V) = \sum_I \det(V_I^T V_I)
\eeq
and we have the inequalities
\beq \label{E:deGua_ineq}
\max_I \sqrt{\det(V_I^T V_I)} \le \sqrt{\det(V^T V)} \le \sum_I \sqrt{\det(V_I^T V_I)}
\eeq
and
\beq \label{E:deGua_ineq2}
{n \choose q}^{-\frac{1}{2}} \sum_I \sqrt{\det(V_I^T V_I)} \le \sqrt{\det(V^T V)}
\eeq
where $V_I$ is the square matrix $\rho_I V$ and the sums are over all subsets $I \subseteq \{ 1, \ldots, n \}$ with $q$ elements.
\end{lemma}

Before proving this lemma, we note that (\ref{E:deGua_eq}) implies (and is essentially equivalent to) the following theorem.  This theorem is a generalization of both Pythagoras' and de Gua's theorems (see \cite[p. 207]{Wells91}, \cite[p. 517]{OsgoodGraustein50} or \cite[p. 21]{Bhatia97}), for when $q = 1$ and $C$ is a line segment then (\ref{E:deGua_simple}) is Pythagoras' theorem, and when $q = n-1$ and $C$ is a $q$-dimensional simplex with vertices on the co-ordinate axes then (\ref{E:deGua_simple}) is de Gua's theorem.

\begin{theorem}
\label{T:deGua}
Let $C$ be a bounded and closed subset of a $q$-dimensional plane in $n$-dimensional Euclidean space $\R^n$.  Then
\beq \label{E:deGua_simple}
\volqeuclsq(C) = \sum_I \volqeuclsq(C_I)
\eeq
where $\volqeuclsq$ is the square of the $q$-dimensional Euclidean volume (i.e., Hausdorff measure), the sum is over all subsets $I \subseteq \{ 1, \ldots, n \}$ with $q$ elements and $C_I = \rho_I(C)$ is essentially the orthogonal projection of $C$ onto the $q$-dimensional plane $\{ \xi \in \R^n \mid \xi_i = 0 \mbox{ if } i \not\in I \}$.
\end{theorem}

\begin{proof}
Let $C \subseteq \R^n$ be any bounded and closed set contained in the column space $\col V$ of some full-rank $n \times q$ matrix $V$ and let $C_I = \rho_I(C)$, as in the statement.
In general, if $W$ is an $m \times q$ matrix then $\volqeuclsq(W(K)) = \det(W^T W) \volqeuclsq(K)$ for any $K \subseteq \R^q$, where $W(K)$ is the image of $K$ under the linear map $x \mapsto W x$.  This follows from the relationship between Gram determinants and the volumes of parallelepipeds \cite[p. 20]{Bhatia97}.
So choosing $K \subseteq \R^q$ so that $V(K) = C$ we have $\volqeuclsq(C) = \det(V^T V) \volqeuclsq(K)$ and $\volqeuclsq(C_I) = \det(V_I^T V_I) \volqeuclsq(K)$, since $V_I(K) = \rho_I V(K) = \rho_I (C) = C_I$.  Multiplying both sides of (\ref{E:deGua_eq}) by $\volqeuclsq(K)$ therefore proves the theorem.
\end{proof}

We now return to Lemma \ref{L:deGua}.

\begin{proof}[Proof of Lemma \ref{L:deGua}]
See \cite[\S I.5]{Bhatia97} for the basic facts about the exterior algebra of a vector space used in this proof.

If $V$ is any $n \times q$ matrix then let $v_1, \ldots, v_q \in \R^n$ be its columns.
Let $\bigwedge^q \R^n$ be the $q^{th}$ exterior power of $\R^n$ (also known as the $q^{th}$ antisymmetric tensor power of $\R^n$) endowed with the inner product given by
$$\langle a_1 \wedge \ldots \wedge a_q, b_1 \wedge \ldots \wedge b_q \rangle = \det[a_i \cdot b_j]$$
on decomposable elements of $\bigwedge^q \R^n$, where $[a_i \cdot b_j]$ is the matrix with $(i,j)^{th}$ element equal to the Euclidean inner product $a_i \cdot b_j$ of $a_i$ and $b_j$.  Then the corresponding squared norm of $v_1 \wedge \ldots \wedge v_q$ is
\begin{equation}\label{E:norm}
\| v_1 \wedge \ldots \wedge v_q \|^2 = \det[v_i \cdot v_j] = \det (V^T V).
\end{equation}

Now, since $V = [v_1 | \ldots | v_q]$, $v_k = \sum_{j=1}^n v_{jk} e_j$ where $v_{jk}$ is the $(j,k)^{th}$ entry of $V$ and $e_1, \ldots, e_n$ is the standard basis for $\R^n$.  So
\begin{eqnarray}
v_1 \wedge \ldots \wedge v_q
&=& \sum_{j_1, \ldots, j_q} v_{j_1 1} \ldots v_{j_q q} \,  e_{j_1} \wedge \ldots \wedge e_{j_q} \nonumber \\
&=& \sum_{i_1 < \ldots < i_q} \left( \sum_{\sigma \in S_q} \sign(\sigma) v_{\sigma(i_1) 1} \ldots v_{\sigma(i_q) q} \right)  \,  \, e_{i_1} \wedge \ldots \wedge e_{i_q} \mbox{ where $j_k = \sigma(i_k)$}\nonumber \\
&=& \sum_{i_1 < \ldots < i_q} (\det V_I) \,  e_{i_1} \wedge \ldots \wedge e_{i_q}
\label{E:coeff}
\end{eqnarray}
where $I = \{i_1, \ldots, i_q\}$, $S_q$ is the symmetric group on $q$ symbols and $\sign(\sigma)$ is $1$ if the permutation $\sigma \in S_q$ is even and $-1$ if it is odd.

Note that all $e_{i_1} \wedge \ldots \wedge e_{i_q}$ for $1 \le i_1 < \ldots < i_q \le n$ form an orthonormal basis for $\bigwedge^q \R^n$, and (\ref{E:coeff}) gives $v_1 \wedge \ldots \wedge v_q$ in terms of this basis.  But Pythagoras' theorem for a finite-dimensional inner product space says that any vector has a squared norm equal to the sum of the squares of its coefficients with respect to any orthonormal basis.  So by (\ref{E:coeff}) and Pythagoras' theorem for $\bigwedge^q \R^n$ we have
\begin{eqnarray}
\| v_1 \wedge \ldots \wedge v_q \|^2 = \sum_{i_1 < \ldots < i_q} (\det V_I)^2  \label{E:pythagoras}.
\end{eqnarray}
Combining this with (\ref{E:norm}) and $(\det V_I)^2 = \det(V_I^T V_I)$ then gives (\ref{E:deGua_eq}).

The left-hand inequality in (\ref{E:deGua_ineq}) follows from (\ref{E:deGua_eq}) and the fact that $\det(V_I^T V_I) = (\det V_I)^2 \ge 0$. To prove the other inequality, note from (\ref{E:coeff}) and (\ref{E:pythagoras}) that the norm $\| \cdot \|$ is the $\ell^2$ norm on $\bigwedge^q \R^n$ corresponding to the basis $e_{i_1} \wedge \ldots \wedge e_{i_q}$ and the $\ell^1$ norm corresponding to this basis is
$$ \| v_1 \wedge \ldots \wedge v_q \|_{\ell^1} \defeq \sum_{i_1 < \ldots < i_q} | \det V_I |
= \sum_{i_1 < \ldots < i_q} \sqrt{\det(V_I^T V_I)}. $$
Therefore the right-hand inequality in (\ref{E:deGua_ineq}) follows from the fact that the $\ell^2$ norm is always less than or equal to the $\ell^1$ norm (as is trivial to prove for finite dimensional spaces, since if $x \in \R^m$ then $\| x \|_{\ell^1}^2 = (\sum_i |x_i|)^2 \ge \sum_i |x_i|^2 = \| x \|^2_{\ell^2}$) and the inequality (\ref{E:deGua_ineq2}) follows from the fact that $\| x \|_{\ell^1} \le  \sqrt{m} \, \| x \|_{\ell^2}$ if $x \in \R^m$ (which can be easily proved with the Cauchy-Schwarz inequality).
\end{proof}

Now, since the branch of $\arcsin$ in (\ref{E:phi}) has domain $(-1,1)$ and range $(-\frac{\pi}{2},\frac{\pi}{2})$,
$$ \frac{\partial \phi_i}{\partial \beta_j}
= \frac{1}{\sqrt{1-\tanh^2 \frac{x_i \beta}{2}}} \left( \frac{1}{\cosh^2 \frac{x_i \beta}{2}} \right) \frac{x_{ij}}{2}
= \frac{x_{ij}}{2 \cosh \frac{x_i \beta}{2}}.$$
Therefore, the Jacobian matrix $J(\beta)$ of $\phi$ at $\beta$ is
\begin{equation}\label{E:Jbeta}
J(\beta) = M(\beta) X
\end{equation}
where $M(\beta)$ is the $n \times n$ diagonal matrix with $i^{th}$ diagonal element $(2 \cosh \frac{x_i \beta}{2})^{-1}$.  As a check on this formula, it is easy to see that substituting (\ref{E:Jbeta}) into Lemma \ref{L:pullback} and using Lemma \ref{L:eucl_cube} gives the same result as Lemma \ref{L:FImatrix_logreg}.

For any $I \subseteq \{ 1, \ldots, n \}$ with $q$ elements, let $X_I = \rho_I X$ be the square matrix obtained from $X$ by deleting all rows of $X$ except those with indices in $I$ and let $\phi_I = \rho_I \phi$ be the projection of $\phi$ onto co-ordinates $i_1, \ldots, i_q$.  We say $\phi_I$ is a local diffeomorphism if it is smooth (infinitely differentiable) and the determinant of its Jacobian matrix is nowhere zero.

\begin{lemma}
\label{L:monotonic}
For any $I \subseteq \{ 1, \ldots, n \}$ with $q$ elements, $\phi_I: \R^q \to (-\frac{\pi}{2},\frac{\pi}{2})^q$ is either injective and a local diffeomorphism or else there is some non-zero $v \in \R^q$ so that $X_I v = 0$ and $\phi_I$ is constant in the direction of $v$, i.e. $\phi_I(\beta + tv) = \phi_I(\beta)$ for all $\beta \in \R^q$ and $t \in \R$.
Since $X$ has full rank, this implies $\phi$ is injective.
\end{lemma}

When $q=1$, this lemma says that each $\phi_i$ is either constant or strictly monotonic.

\begin{proof}[Proof of Lemma \ref{L:monotonic}]  Let $M_I(\beta) \defeq \rho_I M(\beta) \rho_I^T$ be the matrix obtained from $M(\beta)$ by deleting all rows and columns except those with indices in $I$ and let $J_I(\beta)$ be the Jacobian matrix of $\phi_I$ at $\beta \in \R^q$.  Since $\rho_I$ is linear and constant in $\beta$, $J_I(\beta) = \rho_I J(\beta)$, so by (\ref{E:Jbeta}) we have
\begin{equation}\label{E:Jphi}
J_I(\beta)= \rho_I J(\beta) = \rho_I M(\beta) X = \rho_I M(\beta) \rho_I^T \rho_I X = M_I(\beta) X_I
\end{equation}
where $\rho_I M(\beta) = \rho_I M(\beta) \rho_I^T \rho_I$ holds because $M(\beta)$ is diagonal.

We now consider two cases for $\det X_I$.  If $\det X_I = 0$ then there exists some $v \in \R^q$ so that $X_I v = 0$.  So by (\ref{E:Jphi}), $J_I(\beta) v = 0$ for all $\beta$, i.e. for any $i \in I$, the derivative $v \cdot \nabla \phi_i$ of $\phi_i$ in the direction of $v$ is zero for all $\beta$.  So each $\phi_i$ is constant in the direction of $v$, hence $\phi_I(\beta + tv) = \phi_I(\beta)$ for all $\beta$ and $t \in \R$.

If $\det X_I \not= 0$ then by (\ref{E:Jphi}) and the fact that $\det M_I(\beta) > 0$ everywhere, $\det J_I(\beta) \not= 0$ for all $\beta$, so $\phi_I$ is a local diffeomorphism.  To show that $\phi_I$ is injective, let $\alpha, \beta \in \R^q$ with $\alpha \not= \beta$ be given, and we will show that $\phi_I(\alpha) \not= \phi_I(\beta)$.  Define $\gamma:\R \to \R^q$ by $\gamma(t) = \phi_I(t \alpha + (1-t) \beta)$ for any $t \in \R$ and let $\dot{\gamma}$ be the velocity of this path.  Let $w = X_I (\alpha - \beta)$ and note that this is non-zero since $\det X_I \not= 0$ by assumption.
Writing $J_I(t \alpha + (1-t) \beta)$ for $J_I$ evaluated at $t \alpha + (1-t) \beta$, and similarly for $M_I$, by the chain rule we have
$$\dot{\gamma} = J_I(t \alpha + (1-t) \beta) (\alpha - \beta) = M_I(t \alpha + (1-t) \beta) X_I(\alpha - \beta)
= M_I(t \alpha + (1-t) \beta) w$$
so $w^T \dot{\gamma}(t) = w^T M_I(t \alpha + (1-t) \beta) w > 0$ since $M_I$ is positive definite everywhere.  But
$$w^T (\phi_I(\alpha) - \phi_I(\beta)) = w^T \int_0^1 \dot{\gamma}(t) \, dt
= \int_0^1 w^T M_I(t \alpha + (1-t) \beta) w \, dt > 0 $$
so $\phi_I(\alpha) \not= \phi_I(\beta)$, and hence $\phi_I$ is injective.

Now, since $X$ is full-rank, there exists some $I$ with $\det X_I \not=0$.  Therefore the results just proved show that $\phi_I$ and hence $\phi$ is injective.
\end{proof}

By Lemma \ref{L:pullback}, $\phi$ is a local isometry onto its image.  This does not, in itself, imply that $\vol(\SX)$ is the volume of the image of $\phi$ (e.g., consider a function which winds a line around a circle).  However, as a consequence of the injectivity of $\phi$ just proven, we have the following.

\begin{lemma}
\label{L:volembed}
$\vol(\SX)$ is the $q$-dimensional Euclidean volume (i.e., Hausdorff measure) of the subset $\phi(\R^q)$ inside the Euclidean cube $\Xi$ of side-length $\pi$.
\end{lemma}

\begin{proof}
By Lemma \ref{L:pullback}, $\phi$ is a local isometry onto its image, so if $J(\beta)$ is the Jacobian of $\phi$ at $\beta$ (as above) then
\begin{eqnarray*}
\vol(\SX) &=&  \int_{\R^q} \sqrt{\det(X^T D_{X \beta} \, X)} \, d\beta \mbox{ by definition}\\
&=&  \int_{\R^q} \sqrt{\det(J(\beta)^T J(\beta))} \, d\beta \mbox{ by (\ref{E:Jbeta})} \\
&=&  \int_{\phi(\R^q)} \sqrt{\det(g_\Xi)} \, d\xi \mbox{ by Lemmas \ref{L:pullback} and  \ref{L:monotonic}} \\
&=&  \volqeucl(\phi(\R^q)) \mbox{ by definition}
\end{eqnarray*}
where $g_\Xi = I$ is the Euclidean metric on $\Xi$ and $\volqeucl(\phi(\R^q))$ is the $q$-dimensional Euclidean volume (i.e., $q$-dimensional Hausdorff measure) of $\phi(\R^q) \subseteq \Xi$.
\end{proof}

We are now ready to prove our main volume bounds.  For $c,l \in \R^n$, define $\mbox{Box}(c,l) \defeq \{ \xi \in \R^n \mid |\xi_i - c_i| < \frac{1}{2}l_i \}$.  For a Borel-measurable set $U \subseteq \R^q$, let
$$ \vol(\SX |U) \defeq \int_U \sqrt{\det(X^T D_{X \beta} \, X)} \, d\beta$$
be the contribution of volume from $U$ to $\vol(\SX)$.

\begin{theorem}
\label{T:vol_bounds}
Let $U \subseteq \R^q$ be a Borel measurable set.  If $\phi(U) \subseteq \mbox{Box}(c,l)$ for some $c,l \in \R^n$ then
$$ \vol(\SX |U) \le \sum_I\prod_{i \in I}l_i$$
where the sum is over all subsets $I \subseteq \{ 1, \ldots, n \}$ with $q$ elements.  If there exists some $c,l \in \R^n$ (possibly different from those above) and some $I$ so that $\phi_I(U) \supseteq \rho_I(\mbox{Box}(c,l))$ then
$$ \vol(\SX |U) \ge \prod_{i \in I} l_i.$$
\end{theorem}

\begin{proof}
Let $\phi: \R^q \to \Xi$ be as in (\ref{E:phi}) and let $J(\beta)$ be the Jacobian matrix of $\phi$.  As in the proof of Lemma \ref{L:monotonic}, since $\rho_I$ is linear and constant in $\beta$, $J_I(\beta) \defeq \rho_I J(\beta)$ is the Jacobian matrix of $\phi_I(\beta) \defeq \rho_I \phi(\beta)$.
To establish the upper bound on $\vol(\SX|U)$, we have
\begin{eqnarray*}
\vol(\SX|U) &=&  \int_{U} \sqrt{\det(X^T D_{X \beta} \, X)} \, d\beta \mbox{ by definition}\\
&=&  \int_{U} \sqrt{\det(J(\beta)^T J(\beta))} \, d\beta \mbox{ by (\ref{E:Jbeta})} \\
&\le&  \sum_I \int_{U} \sqrt{ \det(J_I(\beta)^T J_I(\beta))} \, d\beta \mbox{ by (\ref{E:deGua_ineq}) with $V = J(\beta)$} \\
&=&  \sum_I \volqeucl(\phi_I(U)) \mbox{ by Lemma \ref{L:monotonic}} \\
&\le&  \sum_I \volqeucl(\rho_I(\mbox{Box}(c,l))) \mbox{ if $\phi(U) \subseteq \mbox{Box}(c,l)$.} \\
&=&  \sum_I \prod_{i \in I}l_i.
\end{eqnarray*}
For the lower bound, if $I$ is such that $\phi_I(U) \supseteq \rho_I(\mbox{Box}(c,l))$ then
\begin{eqnarray*}
\vol(\SX|U) &=&  \int_{U} \sqrt{\det(J(\beta)^T J(\beta))} \, d\beta \mbox{ by (\ref{E:Jbeta}), as above} \\
&\ge&  \int_{U} \sqrt{ \det(J_I(\beta)^T J_I(\beta))} \, d\beta \mbox{ by (\ref{E:deGua_ineq}) with $V = J(\beta)$} \\
&=&  \volqeucl(\phi_I(U)) \mbox{ by Lemma \ref{L:monotonic}} \\
&\ge&  \prod_{i \in I}l_i \mbox{ by the above assumption that $\phi_I(U) \supseteq \rho_I(\mbox{Box}(c,l))$.}
\end{eqnarray*}
\end{proof}

We can now prove Theorem \ref{T:volSXbound}, which states that $ \pi^q \le \vol(\SX) \le {n \choose q} \pi^q.$

\begin{proof}[Proof of Theorem \ref{T:volSXbound}.]
For the upper bound, apply Theorem \ref{T:vol_bounds} with $U=\R^q$, $c_i = 0$ and $l_i = \pi$.

For the lower bound, since $X$ is full-rank, there is some $I$ so that $X_I$ is non-singular.  But then $X_I$ is a design matrix for the saturated model for $q$ binary observations.  Therefore the image of $\phi_{X_I} = \phi_I = \rho_I \phi$ is the cube $(-\pi/2,\pi/2)^q$, since the saturated model is unique up to reparameterisation and it obviously has this image under $\phi_{X_I}$ if $X_I$ is the identity.  So if $U=\R^q$, $c_i = 0$ and $l_i = \pi$ (as above) then $\phi_I(U) \supseteq \rho_I(\mbox{Box}(c,l))$, so applying Theorem \ref{T:vol_bounds} completes the proof.
\end{proof}

Note that the bounds of Theorem \ref{T:volSXbound} are sharp, at least when $q=1$, since the lower bound is realised by $X = [1 \,\,\, 0 \, \ldots \, 0]^T$ and the upper bound is approached by $X = [t \,\,\, t^2 \, \ldots \, t^n]^T$ as $t \to 0$ (consider the image of $\phi$ and use Theorem \ref{T:vol_bounds}).

We now have the following refinement of Theorem \ref{T:vol_bounds}, which shows that the lower bound of Theorem \ref{T:vol_bounds} is only realised by highly degenerate design matrices.

\begin{theorem}
\label{T:vol_bounds_generic}
If $X$ is any $n \times q$ matrix then
$$ \frac{N_1 \pi^q}{\sqrt{{n \choose q}}} \le \vol(\SX) \le N_1 \pi^q$$
where $N_1$ is the number of subsets $I \subseteq \{ 1, \ldots, n\}$ with exactly $q$ elements for which $\det X_I \not= 0$, and $X_I = \rho_I X$.  In particular, if $X$ is generic then $N_1 = {n \choose q}$ so
$$ \vol(\SX) \ge \pi^q \sqrt{{n \choose q}}.$$
\end{theorem}

\begin{proof}
Let $\phi: \R^q \to \Xi$ be as in (\ref{E:phi}) and let $J(\beta)$ be the Jacobian matrix of $\phi$.  As in the proof of Theorem \ref{T:vol_bounds}, $J_I(\beta) \defeq \rho_I J(\beta)$ is the Jacobian matrix of $\phi_I(\beta) \defeq \rho_I \phi(\beta)$.
To establish the lower bound on $\vol(\SX)$, we have
\begin{eqnarray*}
\vol(\SX) &=&  \int_{\R^q} \sqrt{\det(X^T D_{X \beta} \, X)} \, d\beta \mbox{ by definition}\\
&=&  \int_{\R^q} \sqrt{\det(J(\beta)^T J(\beta))} \, d\beta \mbox{ by (\ref{E:Jbeta})} \\
&\ge&  {n \choose q}^{-\frac{1}{2}} \sum_I \int_{\R^q} \sqrt{ \det(J_I(\beta)^T J_I(\beta))} \, d\beta \mbox{ by (\ref{E:deGua_ineq2}) with $V = J(\beta)$} \\
&=&  \frac{N_1 \pi^q}{\sqrt{{n \choose q}}}
\end{eqnarray*}
since the integral $\int_{\R^q} \sqrt{ \det(J_I(\beta)^T J_I(\beta))} \, d\beta$ is $0$ if $\det X_I = 0$, by (\ref{E:Jbeta}), and is $\pi^q$ if $\det X_I \not= 0$, since then $X_I$ is a design matrix for the saturated model for $q$ binary observations and the integral is its volume.

The upper bound is proved similarly, though based on (\ref{E:deGua_ineq}) rather than (\ref{E:deGua_ineq2}).
\end{proof}

\section{$\vol(\SX)$ is a discontinuous function of $X$}
\label{S:discont}

Let $X$ be the full-rank, $q \times n$ design matrix of a logistic regression model $\SX$ and let $\phi: \R^q \to \Xi$ be the isometric embedding of $\SX$ into the Euclidean cube $\Xi$ given by (\ref{E:phi}).  In this section, we will show that $\vol(\SX)$ is a discontinuous function of $X$ (though we will see in Theorem \ref{T:volcont} that $\vol(\SX)$ is continuous at generic $X$, and in Section \ref{S:nongeneric} we will explicitly describe the discontinuities at non-generic $X$).  This makes it unlikely that any closed-form expression for the volume exists in general, but it has interesting consequences when $\vol(\SX)$ is interpreted as a measure of model complexity (see Section \ref{S:compl}).

When $q=n$, there is only one logistic regression model up to reparameterisation, so $\vol(\SX)$ is trivially continuous in this case.  But in all other cases we have the following.

\begin{lemma}
\label{L:discontvol}
$\vol(\SX)$ is a discontinuous function of $X$ for all $q$ and $n$ with $q < n$.
\end{lemma}

\begin{proof}
Let $X$ be the $n \times q$ matrix $X = [I_q \; 0]^T$ consisting of the $q\times q$ identity matrix $I_q$ followed by $n-q$ rows of zeroes.  Then $\vol(\SX) = \pi^q$, but there are generic design matrices $Z$ arbitrarily close to $X$, and these satisfy $\vol(\SZ) \ge \pi^q \sqrt{{n \choose q}}$ by Theorem \ref{T:vol_bounds_generic}.
\end{proof}

We can illustrate how this discontinuity arises as follows (see Figure \ref{F:worms}).  Let $q=1$ and $n =2$, so $X$ is a column matrix with entries $x_1$ and $x_2$, then fix $x_2 = 1$ and consider the limit $x_1 \to 0$.  When $x_1 = 0$, $\phi_1(\beta) = 0$ and $\phi_2(\beta)$ ranges between $-\pi/2$ and $\pi/2$, so  $\vol(\SX) = \pi$ by Lemma \ref{L:volembed}.  But when $x_1 > 0$ then $\phi(\beta) \to \pm \xi$ as $\beta \to \pm \infty$, where $\xi = (\pi/2,\pi/2)$, so $\vol(\SX) \ge d(\xi, -\xi) = \sqrt{2}\pi$.

\begin{figure}[tb]
\centering
\includegraphics[width=10cm]{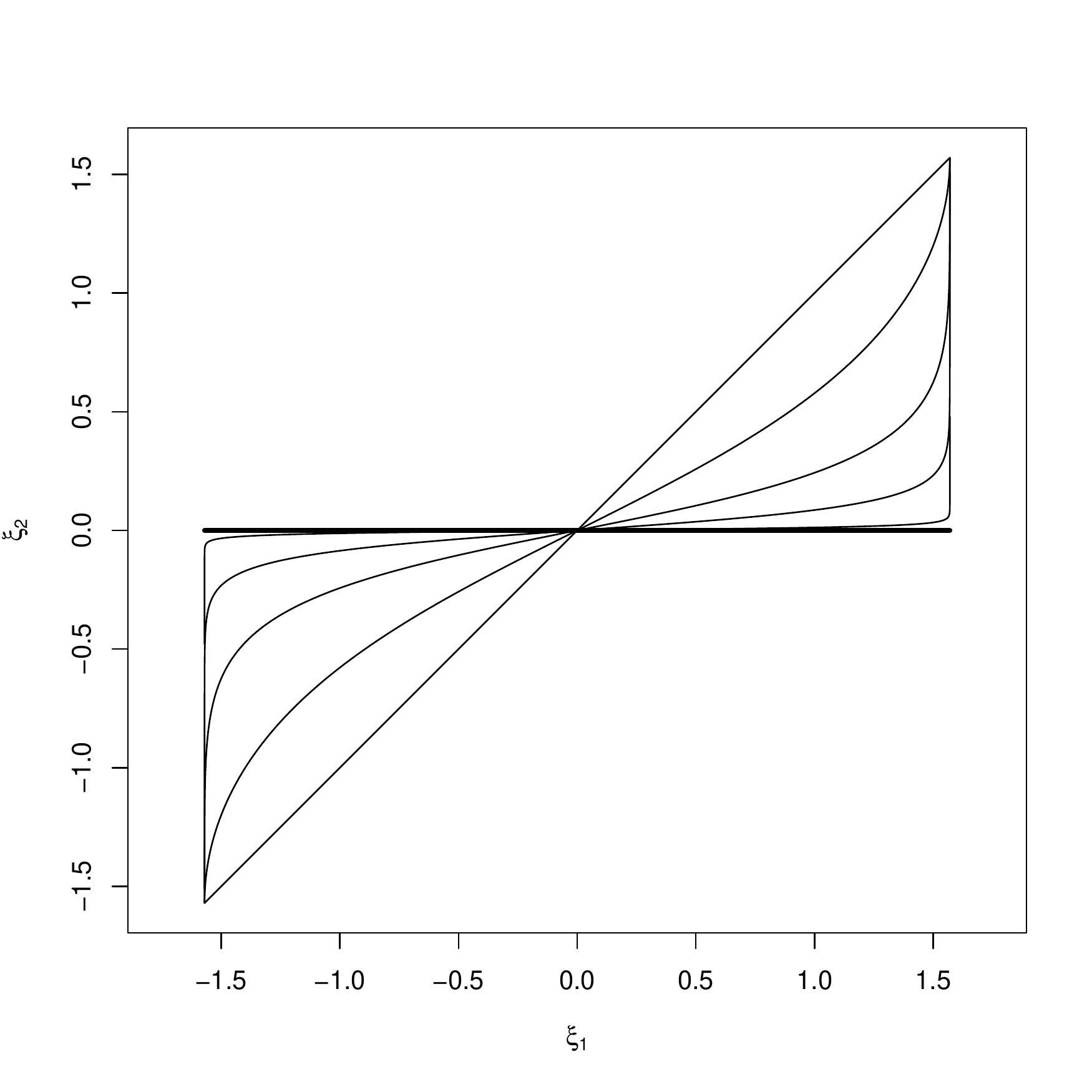} \\
\caption{The images of isometric embeddings of logistic regression models $\SX$ into the Euclidean square $\Xi$ when $q=1$ and $n=2$, for $X = [x_1 \,\,\, 1]^T$ with $x_1 = 1, 0.5, 0.2, 0.07, 0.01$ (thin lines) and $x_1 = 0$ (thick horizontal line).}
\label{F:worms}
\end{figure}

The definition (\ref{E:vol_defn}) expresses $\vol(\SX)$ as the integral over $\R^q$ of a continuous function of $\beta$ and $X$, so it might seem that this would guarantee that $\vol(\SX)$ is continuous in $X$.  This would be true if the integral were over a compact (bounded and closed) domain in $\R^q$, but this argument fails because $\R^q$ is not compact.  For example, $\int_0^R \lambda \exp(-\lambda t) dt$
is continuous as $\lambda$ approaches $0$ from above for any finite $R >0$ but not if $R=\infty$.  However, in Section \ref{S:duality} we will show that the integral (\ref{E:vol_defn}) can effectively be restricted to a fixed compact domain for all design matrices close to a given, generic $X$, so the above argument will then imply continuity at generic $X$.

\section{Volume as a measure of complexity in model selection}
\label{S:compl}

In this section, we briefly recall the MDL principle for model selection before deriving the approximate volume criterion of Definition \ref{D:volcrit} and applying this to an image processing problem.  As before, $X$ is an $n \times q$ full-rank matrix with $q \le n$ and $\SX$ is the corresponding logistic regression model.

\subsection{MDL for model selection}
\label{S:MDL}

The MDL principle is a general information-theoretic criterion for the selection of statistical models \cite{BarronEtAl98,Rissanen07}.  The MDL approach is particularly well-behaved for logistic regression models because these models have finite data spaces.

Suppose we are given a countable set of competing parametric models $\mathcal{S}_1, \mathcal{S}_2, \ldots$ for the data $y$, e.g., each $\mathcal{S}_i$ could be a logistic regression model (each with its own design matrix).  Then the MDL principle advocates choosing the model $\mathcal{S}_i$ with the shortest prefix code for $y$ constructed from a distribution which minimizes the maximum regret for $\mathcal{S}_i$ \cite[\S 2.4.3]{Grunwald05}.  It turns out that this means choosing the model with largest normalized maximum likelihood for the observed data $y$ \cite{Shtarkov87}.

In our main case of interest, namely logistic regression, the MDL principle therefore advocates choosing the model $\SX$ with the smallest value of
$$ - \log p(y | \hat{\beta}(y)) + \comp(\SX)$$
where $p(y | \beta)$ is the likelihood for the observed data $y \in \cy \defeq \{0,1 \}^n$ and regression parameter $\beta$, $\hat{\beta}(y)$ is the maximum likelihood estimate of $\beta$ corresponding to $y$ and the {\em parametric complexity} $\comp(\SX)$ of $\SX$ is
$$ \comp(\SX) \defeq \log \left(\sum_{y \in \cy} p(y | \hat{\beta}(y)) \right).$$
Since $\cy$ has $2^n$ elements, calculating $\comp(\SX)$ from this definition is not practical even for moderately large $n$, so instead we use the approximation
\beq \label{E:complexity}
\comp(\SX) \approx - \frac{q}{2}\log 2\pi + \log \vol(\SX)
\eeq
which is valid for large $n$ \cite[eqn. 2.21]{Grunwald05}.  Note that in (\ref{E:complexity}), an $n$ from \cite[eqn. 2.21]{Grunwald05} has been absorbed into our $\vol(\SX)$, since our Fisher information metric is for $n$ observations while that of \cite{Grunwald05} is effectively for $1$ observation, so our metric is $n$ times that of \cite{Grunwald05}.  Note also that $\SX$ satisfies the regularity conditions given in \cite[p. 48]{Grunwald05} for (\ref{E:complexity}) to be valid, because $\SX$ is an exponential family, $\comp(\SX)$ is finite (since $\cy$ is), and $\vol(\SX)$ is finite by Theorem \ref{T:volSXbound}.

\subsection{An approximation to the volume}
\label{S:approxvol}

Lemma \ref{L:discontvol} says that $\vol(\SX)$ is a discontinuous function of $X$, so it seems unlikely that there exists a closed-form expression for $\vol(\SX)$ which is valid for all $X$ (though such an expression might exist for generic $X$).  So in this section, we derive an approximation to the volume.  We begin by recalling the following definition, which was given briefly in Section \ref{S:overview}.

\begin{defn}[Generic]
\label{D:generic}
An $n \times q$ matrix with $q \le n$ is {\em generic} if any $q$ of its rows are linearly independent.
\end{defn}

Compare this with the condition that the matrix has full rank, which means that some set of $q$ of its rows are linearly independent.  So if $X$ is generic then it has full rank, but the converse is not true (unless $n=q$).

Suppose now that the rows $x_1, \ldots, x_n$ of $X$ and the rows $z_1, \ldots, z_q$ of a $q \times q$ matrix $Z$ are IID random variables so that $X$ has full rank with probability $1$.  This will hold if the covariate distribution is continuous (i.e., has a Lebesgue density) or is continuous apart from an intercept term (i.e., the first component of each $x_i$ is $1$ but the other components form a continuous random variable).  Also note that since the rows of $X$ and $Z$ are IID, the condition that $X$ is full-rank with probability $1$ implies that $X$ and $Z$ are generic with probability $1$.

Then for each $\beta \in \R^q$, by Lemma \ref{L:FImatrix_logreg}, the $(i,j)^{th}$ entry of the Fisher information metric is
\beq
\label{E:metric_IID}
[X^T D_{X\beta} X]_{ij}
= \sum_{k=1}^n  \frac{x_{ki} x_{kj}}{4\cosh^{2}(x_k\beta/2)}
\eeq
which is a sum of $n$ IID random variables.  So by (\ref{E:metric_IID}) and the law of large numbers, for each $\beta \in \R^q$ and large $n$,
\beq
\label{E:avmetric}
[X^T D_{X\beta} X]_{ij}
\approx \E[X^T D_{X\beta} X]_{ij}
= n  \E\left[ \frac{x_{1i} x_{1j}}{4\cosh^{2}(x_1\beta/2)}  \right]
= \frac{n}{q} \E[Z^T D_{Z\beta} Z]_{ij}
\eeq
since $x_1, \ldots, x_n, z_1, \ldots, z_q$ are all identically distributed.  Also, since $X^T D_{X\beta} X$ is continuous in $\beta$ and $X$, the approximation (\ref{E:avmetric}) holds with the same level of accuracy for all $\beta$ in a given compact region of $\R^q$, by the uniform law of large numbers \cite[Lemma 2.4]{NeweyMcfadden94}.  But we will see in the proof of Theorem \ref{T:volcont}, below, that the integral in (\ref{E:vol_defn}) can be restricted to a compact region (up to an arbitrarily small error).  So using the fact that $A \mapsto \sqrt{\det A}$ is a continuous function on the set of positive definite matrices $A$, we have
\begin{eqnarray}
\vol(\SX)
&=& \int_{\R^q} \sqrt{\det(X^T D_{X\beta} X)} d\beta  \mbox{ by definition}\nonumber\\
&\approx&  \left(\frac{n}{q}\right)^{q/2} \int_{\R^q} \sqrt{\det \E \left[Z^T D_{Z\beta} Z \right] } d\beta  \mbox{ by (\ref{E:avmetric}) if $n$ is large} \nonumber \\
&=&  c_q n^{q/2}  \label{E:avvol}
\end{eqnarray}
where the constant $c_q \defeq q^{-q/2} \int_{\R^q} \sqrt{\det \E \left[Z^T D_{Z\beta} Z \right] } d\beta$ does not depend on $n$ but can depend on $q$ and the covariate distribution.

For definiteness, we assume that the covariate distribution and hence $c_q$ is such that the approximation (\ref{E:avvol}) becomes
\beq
\label{E:avvol2}
\vol(\SX) \approx \pi^q \sqrt{{n \choose q}}.
\eeq
This has the asymptotic behaviour given by (\ref{E:avvol}), since ${n \choose q} \sim n^q/q!$ for large $n$ (where we say $a_n \sim b_n$ if $a_n/b_n \to 1$ as $n \to \infty$).  Also, limited computer experiments suggest that $\vol(\SX)$ is a constant multiple of the right-hand side of (\ref{E:avvol2}) for large $n$, where the multiple does not depend on $n$ or $q$ so it does not affect the corresponding model-selection criterion.  Lastly, (\ref{E:avvol2}) gives the minimum volume achieved by {\em generic} design matrices $X$, by Theorem \ref{T:vol_bounds_generic} (recall that $X$ is generic with probability $1$ in this section, and note that the lower volume bound in Theorem \ref{T:volSXbound} is only realised by highly degenerate design matrices).

Since $X$ is here assumed to be generic with probability $1$, we will use the approximation (\ref{E:avvol2}) whenever the design matrix $X$ is generic and, in fact, whenever $X$ has no zero rows (i.e., whenever no row $x_i$ of $X$ has all entries equal to $0$).  However, if $X$ is an $n \times q$ matrix with exactly $n_0$ zero rows then $\vol(\SX) = \vol(\SY)$ where $Y$ is the $(n-n_0) \times q$ matrix obtained from $X$ by deleting the zero rows.  Since $Y$ has $n-n_0$ rows and no zero rows, applying the approximation (\ref{E:avvol2}) to $Y$ and using $\vol(\SX) = \vol(\SY)$ gives
\beq\label{E:volapprox}
\vol(\SX) \approx \pi^q \sqrt{{n-n_0 \choose q}}
\eeq
for any $n \times q$ matrix $X$ with $q\le n$, where $n_0$ is the number of zero rows of $X$.

\subsection{An approximate volume criterion for model selection}
\label{S:approxvolcrit}

We can now use the MDL criterion (Section \ref{S:MDL}) and the approximations (\ref{E:complexity}) and (\ref{E:volapprox}) to obtain a criterion for model selection.  Substituting (\ref{E:volapprox}) into (\ref{E:complexity}) gives
\beq \label{E:complexityapprox}
\comp(\SX) \approx \frac{q}{2}\log \frac{\pi}{2} + \frac{1}{2}\log {n-n_0 \choose q}.
\eeq
So as in Definition \ref{D:volcrit}, our approximate volume criterion advocates choosing the model $\SX$ with the smallest value of
\beq \label{E:approxvolcrit}
- \log p(y | \hat{\beta}(y)) + \frac{q}{2}\log \frac{\pi}{2} + \frac{1}{2}\log {n-n_0 \choose q}
\eeq
where $y$ is the observed data, $\log p(y | \hat{\beta}(y))$ is the maximized log-likelihood and the design matrix of $\SX$ has dimensions $n \times q$ and exactly $n_0$ zero rows.

The main result of \cite{QianField02} shows that this criterion is strongly consistent, in the sense that it will select the correct model almost surely as $n$ goes to infinity, under the weak assumption that all design matrices considered have $n - n_0 \ge \lambda n$ for some fixed $\lambda > 0$.  For as noted above, ${n-n_0 \choose q} \sim (n-n_0)^q/q!$, so $\comp(\SX) \sim (q/2) \log (n-n_0)$ for large $n - n_0$, hence $n - n_0 \ge \lambda n$ implies that $\comp(\SX)$ satisfies the $O(\log(\log n))$ criterion of \cite{QianField02}.  Also, by considering the difference between the right-hand side of (\ref{E:complexityapprox}) and the same expression but with $q-1$ replacing $q$, we see that (\ref{E:complexityapprox}) is increasing in $q$ whenever $n - n_0 \ge 2q$, which by $n - n_0 \ge \lambda n$ is true for all models whenever $n$ is large enough.

For large $n$ and non-sparse models (i.e., those with $n - n_0 \approx n$), the above asymptotic results show that (\ref{E:approxvolcrit}) reduces to the Bayesian information criterion (BIC) \cite{WitEtAl12}.  However, (\ref{E:approxvolcrit}) penalizes sparse models less than the BIC.  We would therefore expect the approximate volume criterion to favour models with sparse design matrices, and hence to be well-suited to situations, such as that of Section \ref{S:imageproc}, where the signal is sparse.

\subsection{Application to image processing}
\label{S:imageproc}

We now present an application of the approximate volume criterion (Definition \ref{D:volcrit} and Section \ref{S:approxvolcrit}) to a simulated image processing problem.  This application was chosen partly because the problem and its solution can be presented graphically, not because we claim our method is particularly suited to image processing.

Consider an image consisting of black and white pixels, as in Figure \ref{F:imageproc}A.  We suppose the image is a noisy version of a black-and-white picture (the signal), where the effect of the noise is to reverse the shade of the pixels $10\%$ of the time, with the noise of different pixels being independent.  We can use logistic regression to de-noise this image as follows.

We interpreted the noisy image as binary data $y \in \{ 0,1 \}^n$ with one observation $y_i$ for each pixel $i$, where $y_i$ is $0$ or $1$ if the pixel is white or black (respectively).  If $A \subseteq \{ 1, \ldots, n \}$ is any subset of the set of all pixels then let $\chi_A$ be the column vector with $i^{th}$ entry equal to $1$ if $i \in A$ or $0$ if $i \not\in A$, so that $\chi_A$ is essentially the characteristic function of $A$.  For the analysis presented here, we generated a design matrix $X$ by specifying that each column of $X$ is of the form $\chi_A$ for some set of pixels $A$ representing a pixelated version of a thickened line segment with a given length, with one of $12$ different orientations and centred at one pixel from a lattice of pixels (which contains approximately one quarter of all pixels).  Since the image consisted of $151 \times 201$ pixels, this gave $q = 86,724$ covariates and $n = 30,351$ observations (note that $q > n$).  Using the LASSO \cite{Tibshirani96,Tibshirani11} implemented in R \cite{RCoreTeam} in the package glmnet \cite{FriedmanEtAl10}, we fitted a path of logistic regression models to the data $y$, with one fitted model for each value of the tuning parameter.  We then chose the tuning parameter using either the approximate volume criterion (Definition \ref{D:volcrit} and Section \ref{S:approxvolcrit}) or by cross-validation, and we plotted the expected values of the two fitted models in Figures \ref{F:imageproc}B and \ref{F:imageproc}C, respectively.  Since our model included an intercept, in the formula (\ref{E:approxvolcrit}) we took $n_0$ to be equal to the number of rows of $X$ which are zero apart from the intercept term.

The approximate volume criterion outperformed cross-validation in terms of mean absolute error ($0.0930$ versus $0.1065$, respectively) though not root-mean-square error ($0.2574$ versus $0.2527$, respectively).  However, from inspection of Figure \ref{F:imageproc}, the estimate based on the approximate volume criterion seems to be a better fit, being very slightly under-fitted to the observed data while the cross-validation estimate is clearly over-fitted.  In addition to this, the approximate volume criterion greatly outperforms cross-validation in terms of calculation speed.

\begin{figure}[tb]
\centering
\includegraphics[width=16cm]{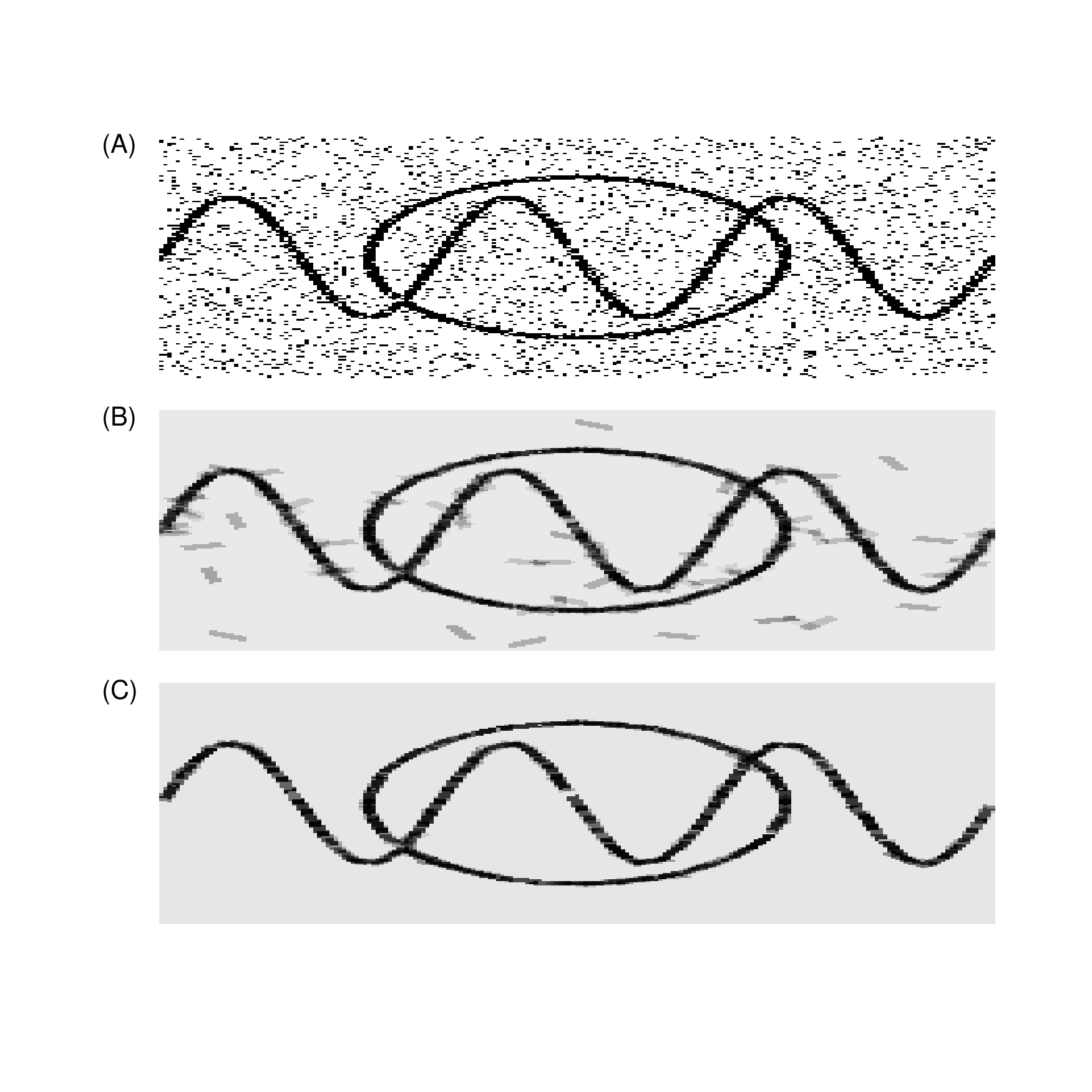} \\
\caption{A noisy black-and-white picture (A) and some de-noised versions of this picture obtained by logistic regression fitted with the LASSO and with tuning parameter chosen by cross-validation (B) or by the approximate volume criterion of Definition \ref{D:volcrit} and Section \ref{S:approxvolcrit} (C), as described in Section \ref{S:imageproc}.}
\label{F:imageproc}
\end{figure}

\section{The behaviour of $\phi(\beta)$ for large $\beta$ and generic $X$}
\label{S:largebeta}

Let $X$ be an $n \times q$ design matrix and let $\phi: \R^q \to \Xi$ be the isometric embedding of the natural parameter space of $\SX$ into the Euclidean cube $\Xi$, as given by (\ref{E:phi}).  In this section we will describe the behaviour of $\phi(\beta)$ for large $\beta$ and generic $X$.  This will allow us to show that $\vol(\SX)$ is continuous at generic $X$ (Section \ref{S:cont}) and that the reparameterisation map between the natural and expectation parameter spaces induces a topological duality (Section \ref{S:duality} and Figure \ref{F:duality}) between certain natural polygonal decompositions on the ideal boundaries of these two spaces (Sections \ref{S:poly_nat} and \ref{S:poly_exp}).

Assume from now on that $X$ is generic (see Definition \ref{D:generic}).

\begin{figure}[tb]
\centering
\includegraphics[width=15cm]{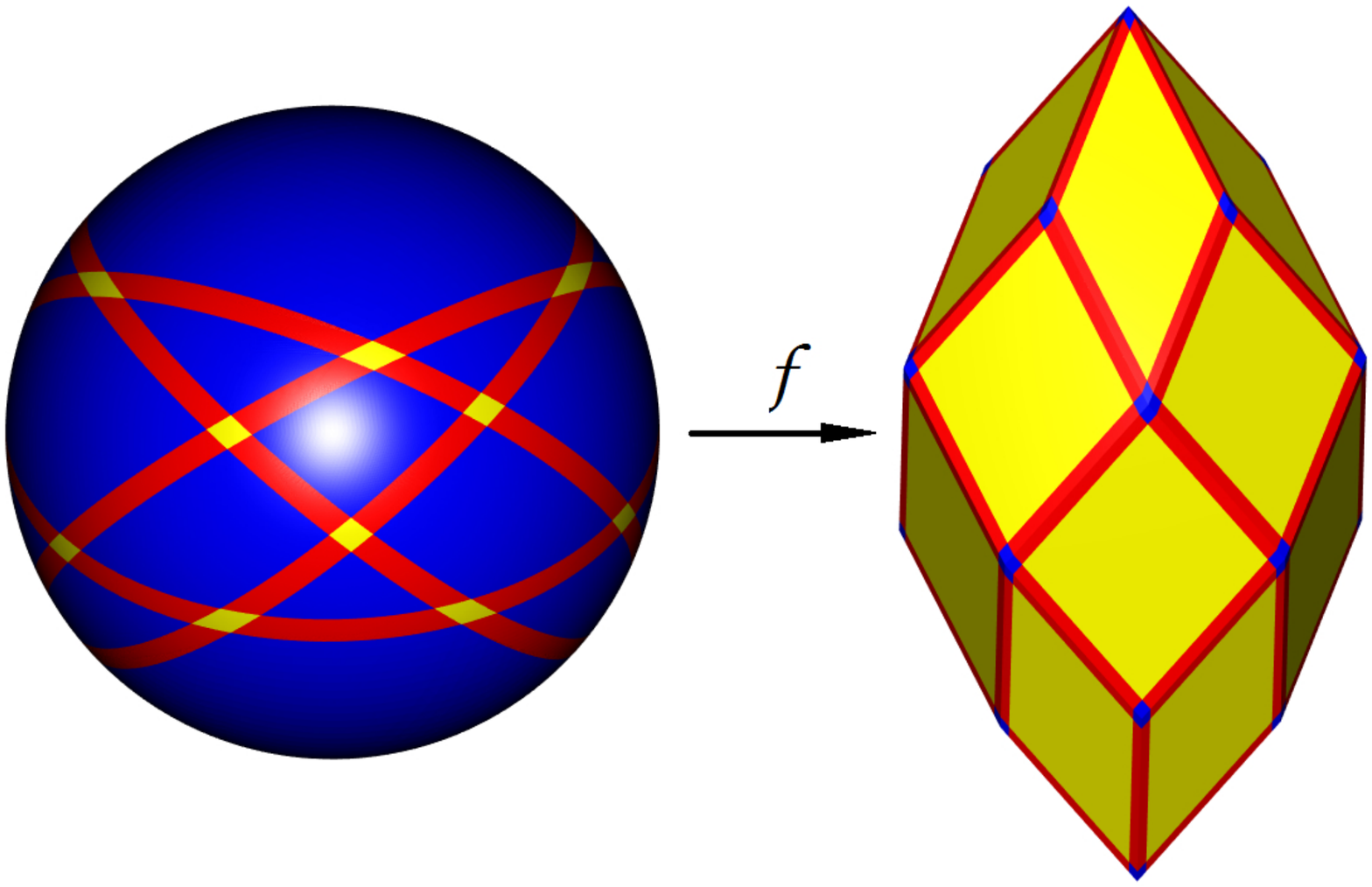} \\
\caption{The sphere $\Sr$ in the natural parameter space (left) and its image (right) under the reparameterisation map $f$ in the expectation parameter space (see Section \ref{S:duality}) when $q=3$ and $n=5$.  The faces $F_{s\delta} \subseteq \Sr$ are shown for $n_s=0$ (blue), $n_s=1$ (red) and $n_s=2$ (yellow), where $s \in S$  has $n_s$ zero components and $\delta=0.5$.  The map $f$ greatly contracts the blue regions and greatly expands the yellow regions (while shrinking the red regions length-wise and stretching them width-wise).  For example, the large blue region at the top of the sphere maps to the small blue region at the very top of the expectation parameter space. }
\label{F:duality}
\end{figure}

\subsection{A polygonal decomposition of the ideal boundary of the natural parameter space}
\label{S:poly_nat}

We now describe a natural polygonal decomposition of the ideal boundary of the natural parameter space $\R^q$ of $\SX$.

For any $r > 0$, let $\Sr$ be the $(q-1)$-dimensional sphere of radius $r$ centred at $0$ in $\R^q$, i.e.,
$\Sr = \{ \beta \in \R^q \mid \beta_1^2 + \ldots + \beta_q^2 = r^2 \}.$
We think of $r$ as being very large, so that $\Sr$ approximates a kind of ideal boundary or `sphere at infinity' of the natural parameter space.

The hyperplanes $\{ \beta \in \R^q \mid x_i \beta = 0 \}$ for $i = 1, \ldots, n$ divide $\Sr$ into spherical polytopes.  More precisely, we can define $\signmap: \Sr \to \{-1,0,1\}^n$ by
$$\signmap(\beta) = (\sign(x_1 \beta), \ldots, \sign(x_n \beta))$$
where, for any $t \in \R$, $\sign(t)$ is $-1$, $0$ or $1$ if $t<0$, $t=0$ or $t>0$ (respectively).  Let $S = \signmap(\Sr) \subseteq \{-1,0,1\}^n$ and, for any $s \in S$, define the corresponding face $F_s$ to be
$$F_s \defeq \signmapinv(s) = \{ \beta \in \Sr \mid \signmap(\beta) = s \}.$$
Each $F_s$ is a (relatively open) spherical polytope, since it is the non-empty set of all $\beta \in \Sr$ which satisfy a set of homogeneous linear equations and inequalities.  Also, the polytopes $F_s$ for all $s \in S$ are clearly disjoint and their union is $\Sr$.  Lastly, since $X$ is generic, $F_s$ is of dimension $q-1-n_s$ (i.e., of codimension $n_s$), where $n_s$ is the number of zero components of $s$ (i.e., the number of indices $i = 1, \ldots, n$ with $s_i = 0$).

We now define a set $F_{s\fdelta} \subseteq \Sr$ which will serve as an approximation to the face $F_s$.
Given any $\fdelta \in (0,\pi/2)$, let $\Delta_\fdelta = 2\arctanh(\sin(\pi/2 - \fdelta))$ so that
$|\phi_i(\beta)| < \pi/2 - \delta$ if and only if $|x_i \beta| < \Delta_\fdelta$,
by (\ref{E:phi}).  Define $\signmapeps: \Sr \to \{-1,0,1\}^n$ by $\signmapeps(\beta) = (\sign_\fdelta(x_1 \beta), \ldots, \sign_\fdelta(x_n \beta))$
where, for any $t \in \R$, $\sign_\fdelta(t)$ is $-1$, $0$ or $1$ if $t<-\Delta_\fdelta$, $|t| \le \Delta_\fdelta$ or $t>\Delta_\fdelta$ (respectively).  Then for any $s \in S$, define
\beq
\label{E:Fseps}
F_{s\fdelta} \defeq \signmapepsinv(s) = \{ \beta \in \Sr \mid \signmapeps(\beta) = s \}.
\eeq
Note that the sets $F_{s\fdelta}$ for all $s \in S$ again partition $\Sr$ into disjoint regions.

The face $F_{s\fdelta}$ is a neighbourhood of $F_s$ in $\Sr$ minus a neighbourhood of the boundary of $F_s$, where these neighbourhoods grow larger with decreasing $\fdelta$.  However, the size of the neighbourhoods do not depend on $r$, so the neighbourhoods can be made arbitrarily small, in relative terms, by making $r$ large.  So for given $\fdelta$, $F_{s\fdelta}$ approximates $F_s$ for large enough $r$.

\subsection{$\vol(\SX)$ is continuous at generic $X$}
\label{S:cont}

In this section we will use the volume bounds of Theorem \ref{T:vol_bounds} to show that $\vol(\SX)$ is continuous at generic $X$, and to suggest a way of numerically calculating $\vol(\SX)$ for such $X$ (see the end of this section).  Note that while the discontinuity of $\vol(\SX)$ (see Lemma \ref{L:discontvol}) makes it unlikely that a closed-form expression for $\vol(\SX)$ exists in general, the following theorem raises the possibility that a simple expression for the volume might exist for generic $X$.

\begin{theorem}
\label{T:volcont}
The volume $\vol(\SX)$ is a continuous function of $X$ at generic $X$.
\end{theorem}

\begin{proof}  Let $B_R = \{ \beta \in \R^q \mid \beta_1^2 + \ldots + \beta_q^2 \le R^2 \}$ be the closed ball in $\R^q$ of radius $R>0$ centred at $0$ (with $R$ chosen below).  Our strategy is to show, for any $n \times q$ matrix $Z$ in a neighbourhood of a given generic $n \times q$ matrix $X$, that the contribution to $\vol(\SZ)$ from outside $B_R$ in the integral (\ref{E:vol_defn}) is arbitrarily small.  This will effectively allow us to restrict the integral (\ref{E:vol_defn}) to the domain $B_R$ for all $Z$ in a neighbourhood of $X$.  Then since $B_R$ is compact (bounded and closed) and the integrand in (\ref{E:vol_defn}) is a continuous function of $X$ and $\beta$, this will imply that $\vol(\SX)$ is continuous at $X$.

So let $X$ be a generic $n \times q$ matrix, as above, and let any $\fdelta \in [0,\pi/2)$ be given.  Then there is some $R>0$ and some neighbourhood $\mathcal{U}$ of $X$ in the space of $n \times q$ real matrices so that if $Z \in \mathcal{U}$ then $Z$ is generic, $S = \signmapZ(\Sr)$ (recall that $S = \signmap(\Sr)$ by definition) and $F_{s\fdelta Z}$ is non-empty for all $s \in S$ and $r \ge R$, where $F_{s\fdelta Z}$ is as in (\ref{E:Fseps}) but with $Z$ replacing $X$.

Then by (\ref{E:Fseps}) and the definition of $\Delta_\fdelta$, if $r > R$ and $s_i \not=0$ then
$| (\phi_Z)_i(\beta) - s_i \pi/2 | < \fdelta$ for all $\beta \in F_{s\fdelta Z}$, where $\phi_Z$ is as in (\ref{E:phi}) but with $Z$ replacing $X$.  So $\phi_Z(F_{s\fdelta Z}) \subseteq \mbox{Box}(c,l)$ where $l = (l_1, \ldots, l_n)$, $c = (c_1, \ldots, c_n)$ and $l_i = \fdelta$, $c_i=s_i (\pi - \fdelta)/2$ if $s_i \not= 0$ or $l_i = \pi$, $c_i=0$ if $s_i = 0$.  Therefore $\phi_Z(U) \subseteq \mbox{Box}(c,l)$, where $U = \cup_{r > R} F_{s\fdelta Z}$ and we recall that $F_{s\fdelta Z}$ is a subset of $\Sr$ so $\cup_{r > R} F_{s\fdelta Z}$ means the union of these subsets for all $r > R$.  So by Theorem \ref{T:vol_bounds},
$\vol(\SZ |U) \le \sum_I\prod_{i \in I}l_i$ where the sum is over all subsets $I \subseteq \{1, \ldots, n\}$ with $q$ elements.  But since $X$ is generic, no more than $q-1$ of the $s_i$ can be zero, hence $\prod_{i \in I}l_i \le \fdelta \pi^{q-1}$ for each $I$ so
$\vol(\SZ |U) \le  \fdelta \pi^{q-1} {n \choose q}$.  Then since $\R^q \setminus B_R = \cup_{s \in S} \cup_{r > R} F_{s\fdelta Z}$, we have
\beq
\label{E:volnegligible}
\vol(\SZ | \R^q \setminus B_R) \le   |S| \fdelta \pi^{q-1} {n \choose q}
\eeq
where $|S|$ is the number of elements of $S$.

Now, because $B_R$ is compact and the integrand in (\ref{E:vol_defn}) is a continuous function of $X$ and $\beta$, $\vol(\SZ | B_R)$ is a continuous function of $Z$ \cite[Theorem 5.6]{Elstrodt96} (this also follows trivially from the fact that the integrand is uniformly continuous on $B_R$).
So after possibly restricting $\mathcal{U}$ to a smaller neighbourhood $\mathcal{U}_\fdelta$ of $X$, if $Z \in \mathcal{U}_\fdelta$ then $| \vol(\SZ | B_R) - \vol(\SX | B_R)| < \fdelta$. Combining this with (\ref{E:volnegligible}) gives
$$| \vol(\SZ) - \vol(\SX)| < \fdelta \left( 1 + |S| \pi^{q-1} {n \choose q} \right)$$
for any $Z \in \mathcal{U}_\fdelta$.

So given any $\feps > 0$, if we set $\fdelta = \feps \left( 1 + |S| \pi^{q-1} {n \choose q} \right)^{-1}$ above then we have shown that there exists a neighbourhood $\mathcal{U}_\fdelta$ of $X$ so that $| \vol(\SZ) - \vol(\SX)| < \feps$ for any $Z \in \mathcal{U}_\fdelta$, hence the theorem is proved.
\end{proof}

The proof of this theorem suggests a way of numerically calculating $\vol(\SX)$ for generic $X$.  For (\ref{E:volnegligible}) gives explicit bounds on the size of $\vol(\SX | \R^q \setminus B_R)$, so (\ref{E:volnegligible}) allows us to choose $R$ and $\delta$ so that $\vol(\SX | \R^q \setminus B_R)$ is smaller than the desired accuracy of the calculation.  Therefore, $\vol(\SX)$ can be approximated by $\vol(\SX | B_R)$ (or $\vol(\SX | U)$ for any $U \supseteq B_R$), and this can be calculated with standard software for integrals over compact domains in $\R^q$.

\subsection{A polygonal decomposition of the ideal boundary of the expectation parameter space}
\label{S:poly_exp}

In this section, we describe the reparameterisation map between the natural and expectation parameter spaces of $\SX$ and then describe the polygonal decomposition of the ideal boundary of the expectation parameter space.

Define $f: \R^q \to \R^q$ by $f(\beta) = X^T h(\phi(\beta))$ where
$h: \overline{\Xi} \to [0,1]^n$ is given by $h = (h_1, \ldots, h_n)$ with $h_i(\xi) = \hf (1 + \sin \xi_i)$ and $\overline{\Xi} = [-\pi/2,\pi/2]^n$ is the closure of $\Xi$.  We claim that $f$ is the reparameterisation map between the natural and expectation parameter spaces of $\SX$.  For by (\ref{E:xi_param}), the restriction of $h$ to the interior $\Xi$ of the closed cube $\overline{\Xi}$ is the reparameterisation map from the Euclidean parameter space of the saturated model to the expectation one.  Therefore $h(\phi(\beta))$ is the expectation parameter of the saturated model corresponding to the natural parameter $\beta$ of $\SX$.  So $h(\phi(\beta))$ is the expected value $\E[y]$ of the sufficient statistic $y$ of the saturated model, where $y$ is distributed according to the natural parameter $\beta$ of $\SX$, hence $f(\beta) = X^T h(\phi(\beta)) = X^T \E[y]= \E[X^T y]$.  Since the logistic regression model $\SX$ is an exponential family with natural parameter $\beta$ and natural sufficient statistic $X^T y$, this shows that $f(\beta)$ is the expectation parameter corresponding to natural parameter $\beta$, proving the claim.

We can now describe the polygonal decomposition of the ideal boundary of the expectation parameter space.  The closure of the expectation parameter space is the convex hull of the finite set $\{ X^T y \mid y \in \{0,1\}^n \}$ of sufficient statistics \cite[Corollary 9.6]{BarndorffNielsen78}, so it is a convex polytope.  Furthermore, since $X$ has full rank, this convex polytope is $q$-dimensional.  Its boundary therefore has a natural cell decomposition into (relatively open) polytopes of dimensions $0, \ldots, q-1$.

We can give a more precise description of this polygonal decomposition in terms of the obvious polygonal decomposition of the boundary of the cube $\overline{\Xi}$.  Let $S$ be as in Section \ref{S:poly_nat} and to any $s \in S$, let $G_s$ be the Euclidean polytope in the boundary of the cube $\overline{\Xi}$ given by
$$ G_s \defeq \{ \xi \in \overline{\Xi} \mid \xi_i = s_i \pi/2 \mbox{ for any $i$ for which $s_i \not= 0$} \}. $$
Note that $G_s$ is of dimension $n - (n - n_s) = n_s$, where $n_s$ is the number of zero components of $s$.

Define $H_s \defeq X^T h(G_s)$.  We claim that $H_s$ is a polygonal face in the boundary of the closure of the expectation parameter space.  To see this, note that $h(G_s)$ is a polygonal face in the boundary of the cube $[0,1]^n$ which is just a translated and re-scaled version of $G_s$.  So since the map $\ex \mapsto X^T \ex$ is linear, $H_s$ is a polytope in $\R^q$, and is equal to the convex hull of its vertices.
But each of these vertices is of the form $H_t = X^T h(G_t)$ where $t \in S$ has $n_t = 0$.
Therefore, $h(G_t) = y \in \{ 0,1 \}^n$ (more properly, $h(G_t) = \{ y \}$), and any $\beta \in F_t$ separates the $0$s and $1$s of $y$ (meaning $x_i \beta > 0$ if $y_i = 1$ and $x_i \beta < 0$ if $y_i = 0$).  Therefore no maximum likelihood estimate corresponding to data $y$ can exist, so $H_t = X^T h(G_t) = X^T y$ cannot lie in (the interior of) the expectation parameter space, by \cite[Corollary 9.6]{BarndorffNielsen78}.  Therefore $H_s$ is a polygonal face in the ideal boundary of the expectation parameter space, as claimed.

Since $X$ is generic, $\ex \mapsto X^T \ex$ is injective on all $k$-dimensional faces in the boundary of the cube $[0,1]^n$ for $k \le q$, so $H_s$ has the same dimension as $G_s$, namely $n_s$.

Lastly, it follows from Corollary \ref{C:boundaryphi}, below, that the closure of the expectation parameter space is obtained by adding $\cup_{s \in S} H_s$ to this space, so every face in the ideal boundary of the expectation parameter space is of the form $H_s$ for some $s \in S$ (though we will not use this fact until after Corollary \ref{C:boundaryphi}).

\subsection{Duality between the polygonal boundary decompositions}
\label{S:duality}

We will now show, for generic $X$, that the reparameterisation map $f$ between the natural and expectation parameter spaces of $\SX$ induces a topological duality between the polygonal decompositions of the ideal boundaries of these two spaces (see Figure \ref{F:duality}).  Under this map, $k$-dimensional faces in the $(q-1)$-dimensional boundary of one space correspond to $(q-1-k)$-dimensional faces in the boundary of the other space, for all $k = 0, \ldots, q-1$.  This highly unusual behaviour is interesting in its own right, but it also has implications for the computation of $\vol(\SX)$.

We will begin by showing that the cell $F_s$ in the ideal boundary of the natural parameter space of $\SX$ approximately corresponds under $\phi$ to the face $G_s$ in the ideal boundary of the Euclidean cube $\Xi$.  Then the duality result described above will follow from the close relationship between $G_s$ and $H_s$ developed in Section \ref{S:poly_exp}.

If $A$ and $B$ are any bounded subsets of the same Euclidean space then the Hausdorff distance $d_H(A,B)$ between $A$ and $B$ is
$$d_H(A,B) \defeq \inf \{ \feps \ge 0 \mid A \subseteq N_\feps(B) \mbox{ and } B \subseteq N_\feps(A) \}$$
where $N_\feps(A) = \{ x \in \E^n \mid \mbox{$\exists a \in A$ so that $d(x,a) < \feps$} \}$ is an $\feps$-neighbourhood of $A$, and similarly for $N_\feps(B)$.

The following theorem says that the image of $F_{s\fdelta}$ under $\phi$ is approximately $G_s$, with the approximation becoming arbitrarily good for $r$ large enough.  This is despite the fact that $F_s$ and $G_s$ have different dimensions in general and the fact that $F_{s\fdelta}$ approximates $F_s$ arbitrarily well for large enough $r$ (recall that $F_{s\fdelta} \subseteq \Sr$ so $F_{s\fdelta}$ depends on $r$).

\begin{theorem}
\label{T:duality}
For any $\feps > 0$, there exists $R > 0$ so that
$$d_H(\phi(F_{s\fdelta}),G_s) < \feps$$
for any $s \in S$ and any $r>R$, where $\fdelta = \feps/q\sqrt{n}$.
\end{theorem}

\begin{proof}
Let $\feps_0 > 0$ be given and let $\fdelta = \feps_0/\sqrt{n}$ (and assume, without loss of generality, that $\feps_0$ is small enough that $\fdelta < \pi/2$).  Choose $R > 0$ so that $F_{s\fdelta}$ is a non-empty set in $\Sr$ for all $s \in S$ and all $r>R$.

By (\ref{E:phi}) and the definition of $F_{s\fdelta}$, if $i$ is such that $s_i \not=0$ then
$| \phi_i(\beta) - s_i \pi/2| < \fdelta$ for all $\beta \in F_{s\fdelta}$.
Therefore $\phi(F_{s\fdelta}) \subseteq N_{\feps_0}(G_s)$.

Now, let $\feps_k = (k+1) \feps_0$.  We will use induction on $k$ to prove $G_s \subseteq N_{\feps_k}(\phi(F_{s\fdelta}))$ for all $s \in S$ with $n_s \le k$, where $n_s$ is the number of components of $s$ which are zero.  For the base case, $k = n_s = 0$ so $G_s$ is a point, hence the fact just proved that $\phi(F_{s\fdelta}) \subseteq N_{\feps_0}(G_s)$ implies $G_s \subseteq N_{\feps_0}(\phi(F_{s\fdelta}))$, here also using $F_{s\fdelta} \not= \emptyset$.  Now, for
$k \in \{ 0, \ldots, q-2\}$, assume the induction hypothesis that $G_s \subseteq N_{\feps_k}(\phi(F_{s\fdelta}))$ for all $s \in S$ with $n_s \le k$.  Our goal is to prove this for $k+1$ so let $s \in S$ be such that $n_s = k+1$.

Dual to the polygonal decomposition of $\Sr$ into faces $F_t$ for $t \in S$ there is a decomposition of $\Sr$ into topological, relatively open polygonal faces $F_t^*$ for $t \in S$, so that the face $F_t^*$ has dimension $n_t$ (while $F_t$ has dimension $q - 1 - n_t$, i.e., codimension $n_t$) and so that the association $F_t \mapsto F_t^*$ reverses inclusions (on the closures of the faces), see \cite[\S 3.4]{Grunbaum03} for related results.

Now, with $s \in S$ such that $n_s = k+1$, as above, define
$$ T_s \defeq \{ t \in S \mid \mbox{$\forall i$, $s_i \not= 0$ implies $t_i \not= 0$} \}.$$
Then $\cup_{t\in T_s} F_t^*$ is the closure of $F_s^*$ (since, for each $t \in T_s$, $F_t$ contains the closure of $F_s$ so $F_t^*$ is contained in the closure of $F_s^*$ by the inclusion-reversing property).  So by choosing a larger $R$ (and hence $r$) if need be, the face $F_s^*$ will lie in $\cup_{t\in T_s} F_{t \delta}$.  So by the induction hypothesis, $G_t \subseteq N_{\feps_k}(\phi(F_{t\fdelta}))$ for all $t\in T_s\setminus\{s\}$.  But the ideal boundaries of the faces $G_s$ and $F_s^*$ are $\partial G_s \defeq \cup_{t\in T_s\setminus\{s\}} G_t$ and $\partial F_s^* \defeq \cup_{t\in T_s\setminus\{s\}} F_t^*$ respectively, so this implies that $\partial G_s \subseteq N_{\feps_k}(\phi(\partial F_s^*))$ and that the topological sphere $\phi(\partial F_s^*)$ is homotopically non-trivial in the $\feps_k$-neighbourhood of the topological sphere $\partial G_s$.

Now, given any $\xi \in G_s$, our goal is to show that there is some $\beta \in F_{s\delta}$ so that $d(\phi(\beta), \xi) < \feps_{k+1}$.  We now consider two cases, $\xi \not\in N_{\feps_k}(\partial G_s)$ and $\xi \in N_{\feps_k}(\partial G_s)$.  Write $\xi = \xi_1$ in the first case.  Then since $\phi(\partial F_s^*)$ is homotopically non-trivial in $N_{\feps_k}(\partial G_s)$, there is some $\beta \in F_s^*$ so that the orthogonal projection of $\phi(\beta)$ onto the span of $G_s$ is $\xi_1$ (essentially by \cite[Th. VI.14.14]{Bredon93}).   Also, $\beta \in F_s^* \cap F_{s\delta}$ since otherwise $\xi_1 \in N_{\feps_k}(\partial G_s)$ by the induction hypothesis.  But we have already shown that $\phi(F_{s\fdelta}) \subseteq N_{\feps_0}(G_s)$, so $d(\phi(\beta), \xi_1) < \feps_0$.  Now consider the second case, that $\xi \in N_{\feps_k}(\partial G_s)$, and write $\xi = \xi_2$.  If $\xi_2 \in G_s$ lies in $N_{\feps_k}(\partial G_s)$ then $\xi_2$ is within $\feps_k$ of a point $\xi_1$ of $G_s$ not lying in $N_{\feps_k}(\partial G_s)$, so $d(\phi(\beta), \xi_2) \leq d(\xi_2, \xi_1) + d(\phi(\beta), \xi_1) < \feps_k + \feps_0 = \feps_{k+1}$.  Hence $G_s \subseteq N_{\feps_{k+1}}(\phi(F_{s\fdelta}))$, so the induction hypothesis is proved.

So by induction, $G_s \subseteq N_{\feps_{q-1}}(\phi(F_{s\fdelta}))$ for all $s \in S$ with $n_s \le q-1$.  But since $X$ is generic and $r>R$, all $s \in S$ have $n_s \le q-1$.  Hence $d_H(\phi(F_{s\fdelta}),G_s) < \feps_{q-1} = q\feps_0$ for all $s \in S$.

So given any $\feps > 0$, choose $\feps_0 = \feps/q$ in the above work to establish the theorem.
\end{proof}

Theorem \ref{T:duality} immediately has the following corollary, which says that the faces $G_s$ form the ideal boundary of  the image of $\phi$.

\begin{corol}
\label{C:boundaryphi}
The closure of $\phi(\R^q)$ is obtained by adding $\cup_{s\in S} G_s$ to $\phi(\R^q)$.
\end{corol}

We now have the following theorem, which says that the image of the face $F_{s\fdelta}$ under the reparameterisation map $f$ is approximately $H_s$.  Since $F_{s\fdelta}$ approximates $F_s$ for large $r$ (in relative terms), this shows that $f$ induces a duality between the polygonal decomposition of the ideal boundary of the natural parameter space and that of the expectation parameter space.

\begin{theorem}
\label{T:duality_exp}
For any $\feps > 0$, there exists $R > 0$ and $\fdelta > 0$ so that
$$d_H(f(F_{s\fdelta}),H_s) < \feps$$
for any $s \in S$ and any $r>R$ (for a generic design matrix $X$).
\end{theorem}

\begin{proof}
This follows by applying the function $\xi \mapsto X^T h(\xi)$ to Theorem \ref{T:duality} and by the fact that this function is continuous on $\overline{\Xi}$.
\end{proof}

Since the vertices of the ideal boundary of the expectation parameter space correspond one-to-one to data vectors $y \in \{0,1\}^n$ for which no maximum likelihood estimate exists, we have the following corollary of the duality just proved in Theorem \ref{T:duality_exp}.

\begin{corol}
The number of data vectors $y \in \{0,1\}^n$ for which no maximum likelihood estimate exists is equal to the number of connected components of
$$\{ \beta \in \R^q \mid \mbox{$x_i \beta \not= 0$ for all $i = 1, \ldots, n$ } \}$$
where we recall that the design matrix $X$ is generic and $x_i$ is its $i^{th}$ row.
\end{corol}

Lastly, this duality (in the form of Theorem \ref{T:duality}), also implies that the contribution to $\vol(\SX)$
is concentrated in constant-width neighbourhoods of certain lines (the lines at the intersection of $q-1$ of the hyperplanes $\{ \beta \in \R^q \mid x_i \beta = 0  \}$).  This fact might be useful when trying to numerically evaluate $\vol(\SX)$ via the integral (\ref{E:vol_defn}).

\section{Volume jumps at non-generic $X$}
\label{S:nongeneric}

In this section we show that the volume $\vol(\SX)$ is discontinuous at every non-generic $X$ which, together with Theorem \ref{T:volcont}, shows show that $\vol(\SX)$ is continuous at $X$ if and only if $X$ is generic.  We also show that the volume jump $\vol(\SZ) - \vol(\SX)$ between the volumes of a non-generic matrix $X$ and a nearby generic matrix $Z$ is $\pi^q$ or larger, and that size of the volume jump reflects the degree of degeneracy of $X$.

Let $X$ be a full-rank, real $n \times q$ matrix.  Define the {\em degree of degeneracy} $N_0$ of $X$ to be the number of subsets $I \subseteq \{ 1, \ldots, n\}$ with exactly $q$ elements for which $\det X_I = 0$, where $X_I$ is the matrix obtained from $X$ by deleting the rows with row numbers not in $I$.  
Note that $X$ is generic if and only if $N_0 = 0$.

Define the {\em minimum volume jump} at $X$ to be
$$ \dmin \defeq \lim_{\epsilon \to 0^+} \inf_{Z \in \mathcal{U}_\epsilon} (\vol(\SZ) - \vol(\SX)) $$
where $\mathcal{U}_\epsilon = \{ Z \in \R^{n \times q} \mid \mbox{ $Z$ is generic and } \| X - Z \|_F < \epsilon \}$ is the set of all generic matrices within a distance $\epsilon > 0$ of $X$ in the space $\R^{n \times q}$ of real $n \times q$ matrices endowed with the Frobenius norm $\| Z \|_F = \sqrt{\tr Z^T Z}$ ($\| Z \|_F^2$ is just the sum of the squares of all the components of $Z$, so this is the Euclidean norm on $\R^{n \times q}$).  Note that the set of generic matrices is an open and dense subset of $\R^{n \times q}$, so $\mathcal{U}_\epsilon$ is non-empty for all $X \in \R^{n \times q}$ and $\epsilon > 0$.  Similarly, define the {\em maximum volume jump} at $X$ to be
$$ \dmax \defeq \lim_{\epsilon \to 0^+} \sup_{Z \in \mathcal{U}_\epsilon} (\vol(\SZ) - \vol(\SX)). $$
Note that if the volume is continuous at $X$ then $\dmin = \dmax = 0$ (and the converse is true in light of the following theorem).

\begin{theorem}
\label{T:voljump_simple}
If $X$ is non-generic then
$$\dmin \ge  \pi^q.$$
Together with Theorem \ref{T:volcont}, this implies that the volume $\vol(\SX)$ is continuous at $X$ if and only if $X$ is generic, and the volume jump at non-generic $X$ is always at least $\pi^q$.  Further,
$$\dmin \ge \frac{N_0 \pi^q}{\sqrt{{n \choose q}}}$$
where $N_0$ is the degree of degeneracy of $X$.
\end{theorem}

\begin{proof}
Given any $\delta > 0$, choose $R > 0$ large enough that $|\vol(\SX) - \vol(\SX | B_R)| < \delta$ where $B_R$ is the ball of radius $R$ centred at the origin in the natural parameter space $\R^q$ (such an $R$ exists by the dominated convergence theorem \cite{Elstrodt96}).  Given any $\epsilon > 0$, let $Z$ be a generic matrix within a given distance $\epsilon > 0$ of $X$.  We want to compare $\vol(\SZ)$ to $\vol(\SX)$, so we start by writing $\vol(\SZ)$ as the sum of two terms:
\beq \label{E:vol_twoparts}
\vol(\SZ) = \vol(\SZ | B_R) + \vol(\SZ | \R^q \setminus B_R).
\eeq

We first claim that the first term on the right-hand side of (\ref{E:vol_twoparts}) is approximately $\vol(\SX)$.  To see this, note that since $B_R$ is compact and the integrand of (\ref{E:vol_defn}) is continuous, $\vol(\SX | B_R)$ is a continuous function of $X$.  So if we let `$\approx$' denote an approximate equality which can be made arbitrarily good by taking $\delta$ and $\epsilon$ small enough, then
\beq \label{E:vol_compactpart}
\vol(\SZ | B_R) \approx \vol(\SX | B_R) \approx \vol(\SX).
\eeq

We next claim that the second term on the right-hand side of (\ref{E:vol_twoparts}) can be approximately bounded below by $\pi^q$.  To see this, we note first that because $X$ is non-generic, there is some subset $I \subseteq \{ 1, \ldots, n \}$ with exactly $q$ elements so that $\det X_I = 0$.  As argued above, $\vol(\SXI |B_R)$ is a continuous function of $X_I$ since $B_R$ is compact, so
\beq
\label{E:vol_SZI}
\vol(\SZI |B_R) \approx \vol(\SXI |B_R) = 0
\eeq
where $\SZI$ is the logistic regression model with $q \times q$ design matrix $Z_I$ and the last step follows because $\det X_I = 0$ so $\vol(\SXI) = 0$.
So letting $J(\beta)$ and $J_I(\beta)$ denote the Jacobian matrices of $\phi_Z$ and $\phi_{Z_I}$ (respectively), we have
\begin{eqnarray}
\vol(\SZ |\R^q \setminus B_R) &=&  \int_{\R^q \setminus B_R} \sqrt{\det(X^T D_{X \beta} \, X)} \, d\beta \mbox{ by definition} \nonumber\\
&=&  \int_{\R^q \setminus B_R} \sqrt{\det(J(\beta)^T J(\beta))} \, d\beta \mbox{ by (\ref{E:Jbeta})} \nonumber\\
&\ge&  \int_{\R^q \setminus B_R} \sqrt{ \det(J_I(\beta)^T J_I(\beta))} \, d\beta \mbox{ by (\ref{E:deGua_ineq}) with $V = J(\beta)$} \nonumber\\
&=& \vol(\SZI) - \vol(\SZI |B_R)  \nonumber \\
&\approx&  \vol(\SZI) \mbox{ by (\ref{E:vol_SZI})} \nonumber \\
&=& \pi^q \mbox{ by Theorem \ref{T:volSXbound}} \label{E:vol_noncompactpart}
\end{eqnarray}
So combining (\ref{E:vol_twoparts}), (\ref{E:vol_compactpart}) and (\ref{E:vol_noncompactpart}) gives $\dmin \ge  \pi^q$.  Therefore $\vol(\SX)$ is discontinuous at non-generic $X$ and the volume jump there is always $\pi^q$ or larger.

To prove the other bound on $\dmin$, let $\mathcal{I}$ be the set of all subsets $I \subseteq \{1, \ldots, n\}$ with exactly $q$ elements for which $\det X_I = 0$.  So $\mathcal{I}$ is non-empty, since $X$ is non-generic, and $\mathcal{I}$ has $N_0$ elements $I$, by the definition of the degree of degeneracy.  Letting $I \subseteq \{1, \ldots, n\}$ be a subset with exactly $q$ elements, if $I \in \mathcal{I}$ then (\ref{E:vol_SZI}) holds (by the same reasoning as above), and if $I \not\in \mathcal{I}$ then
\beq
\label{E:vol_SZI2}
\vol(\SZI |B_R) \approx \vol(\SXI |B_R) \approx \vol(\SXI) = \pi^q = \vol(\SZI)
\eeq
where the last two equalities follow by Theorem \ref{T:volSXbound}
and the second approximate equality holds after perhaps taking a larger $R$.
Then in place of (\ref{E:vol_noncompactpart}) we have the following:
\begin{eqnarray}
\vol(\SZ |\R^q \setminus B_R) &=& \int_{\R^q \setminus B_R} \sqrt{\det(J(\beta)^T J(\beta))} \, d\beta \mbox{ by (\ref{E:Jbeta})} \nonumber\\
&\ge&  {n \choose q}^{-\frac{1}{2}} \sum_I  \int_{\R^q \setminus B_R} \sqrt{ \det(J_I(\beta)^T J_I(\beta))} \, d\beta \mbox{ by (\ref{E:deGua_ineq2}) with $V = J(\beta)$} \nonumber\\
&=& {n \choose q}^{-\frac{1}{2}} \sum_I  \left[\vol(\SZI) - \vol(\SZI |B_R) \right] \nonumber \\
&\approx&  {n \choose q}^{-\frac{1}{2}} \sum_{I \in \mathcal{I}}  \vol(\SZI) \mbox{ by (\ref{E:vol_SZI}) and (\ref{E:vol_SZI2})} \nonumber \\
&=& {n \choose q}^{-\frac{1}{2}} N_0 \pi^q \mbox{ by Theorem \ref{T:volSXbound}} \label{E:vol_noncompactpart2}
\end{eqnarray}
Combining (\ref{E:vol_twoparts}), (\ref{E:vol_compactpart}) and (\ref{E:vol_noncompactpart2}) gives the second bound on $\dmin$ in the statement.
\end{proof}

\section{Conclusions}
\label{S:conc}

This paper studied logistic regression models and their volumes.  Our main result bounds the volume of a logistic regression model and, in particular, implies the novel result that the volume is always finite.  This implies that logistic regression models have proper Jeffreys priors, so the volume can be interpreted as a measure of model complexity in the simplest and most elegant version of the MDL approach.  We gave an approximation to the volume and derived a corresponding model-selection criterion, and as a proof of principle we applied this criterion to an image processing problem.  We also showed that the volume is a continuous function of the design matrix $X$ at generic $X$ but is discontinuous in general.  Our model-selection criterion therefore favours models with sparse design matrices, analogous to the way that $\ell^1$-regularisation favours sparse parameter estimates, though in our case this behaviour arises spontaneously from general principles.

We also proved that the ideal boundaries of the natural and expectation parameter spaces of logistic regression models have natural polygonal decompositions which are topologically dual under the reparameterisation map (see Figure \ref{F:duality}).  The full causes and implications of this extremely unusual behaviour are not clear, however this behaviour does not appear to be a consequence of known dualities for exponential families (e.g., convex conjugation \cite[Ch. 9]{BarndorffNielsen78}), so it might hint at a deeper duality.

Lastly, we proved a generalisation of the classical theorems of Pythagoras and de Gua, which is of independent interest.

\section*{Acknowledgements}

The author would like to thank Enes Makalic and Daniel F. Schmidt for introducing him to the volume as a measure of model complexity and for interesting subsequent discussions.

\bibliographystyle{plain}
\bibliography{S:/Staff/EnesAndDaniel/Bibliography/bibliography}

\end{document}